\documentclass{siamltex}
\usepackage[utf8]{inputenc}
\usepackage[T1]{fontenc}

\usepackage{amssymb,amsmath}
\usepackage{color}
\usepackage{xspace}
\usepackage{enumitem}
\usepackage[colorlinks,citecolor=blue]{hyperref}
\usepackage{ntheorem} % For \renewtheorem
\usepackage[nameinlink,noabbrev,capitalize]{cleveref}
\newcommand{\F}{\mathbb{F}}
\newcommand{\R}{\mathbb{R}}
\newcommand{\Pb}{\eqref{eq:prob}\xspace}
\newcommand{\bdp}{\eqref{bdp}\xspace}
\newcommand{\Pbh}{\eqref{eq:probh}\xspace}
\newcommand{\Pbdh}{\eqref{eq:probdh}\xspace}
\newcommand{\bdph}{\eqref{bdph}\xspace}
\newcommand{\uad}{\mathcal{U}_{\textup{ad}}}
\newcommand{\uadh}{\mathcal{U}_{\textup{ad},h}}
\newcommand{\uh}{\mathcal{U}_h}
\newcommand{\yh}{Y_h}

\newcommand{\dx}{\, \textup{d}x}

\newcommand{\deta}{\, \textup{d}\eta}
\newcommand{\BB}{\mathcal{B}}

\renewcommand{\AA}{\mathcal{A}}

\newcommand\hatbar[1]{\hat{\bar{#1}}}

\newcommand{\abs}[1]{\lvert #1 \rvert}
\newcommand{\bigabs}[1]{\bigl\lvert #1 \bigr\rvert}
\newcommand{\Bigabs}[1]{\Bigl\lvert #1 \Bigr\rvert}
\newcommand{\norm}[1]{\lVert #1 \rVert}
\newcommand{\bignorm}[1]{\bigl\lVert #1 \bigr\rVert}

\newcommand{\weakly}{\rightharpoonup}
\newcommand{\weaklystar}{\stackrel{*}{\weakly}}

\newcommand\ad[1]{a(\cdot,#1)}
\newcommand\apd[1]{a'(\cdot,#1)}
\newcommand\appd[1]{a''(\cdot,#1)}
\newcommand\bd[1]{b(\cdot,#1)}
\newcommand\bpd[1]{b'(\cdot,#1)}
\newcommand\bppd[1]{b''(\cdot,#1)}

\newcommand\bary{y_{\bar u}}
\newcommand\yu{y_u}
\newcommand\zuv{z_{u,v}}
\newcommand\zbaruv{z_{\bar u,v}}

\newcommand\barwv{\bar w_v}
\newcommand\wv{w_v}

\newcommand{\mscLink}[1]{\href{http://www.ams.org/mathscinet/msc/msc2010.html?t=#1}{#1}}

\setenumerate{itemsep=3pt,topsep=3pt}

\def\mathllap{\mathpalette\mathllapinternal}
\def\mathrlap{\mathpalette\mathrlapinternal}

\def\mathllapinternal#1#2{\llap{$\mathsurround=0pt#1{#2}$}}
\def\mathrlapinternal#1#2{\rlap{$\mathsurround=0pt#1{#2}$}}

\newcommand\dual[2]{\langle #1, #2 \rangle}

\newcommand\todo[1]{\textbf{\textcolor{red}{#1}}}

\makeatletter
\if@onethmnum
	\renewtheorem{theorem}{Theorem}
	\renewtheorem{lemma}[theorem]{Lemma}
	\renewtheorem{corollary}[theorem]{Corollary}
	\renewtheorem{proposition}[theorem]{Proposition}
	\renewtheorem{definition}[theorem]{Definition}
\else
	\renewtheorem{theorem}{Theorem}[section]
	\renewtheorem{lemma}[theorem]{Lemma}
	\renewtheorem{corollary}[theorem]{Corollary}
	\renewtheorem{proposition}[theorem]{Proposition}
	\renewtheorem{definition}[theorem]{Definition}
\fi
\makeatother
\newtheorem{remark}[theorem]{Remark}
\title{Second-Order Analysis and Numerical Approximation for Bang-Bang Bilinear Control Problems\thanks{The first author was partially supported by the Spanish Ministerio de Econom\'{\i}a y Competitividad under project MTM2014-57531-P. The second author was partially supported by DFG under grant number Wa 3626/1-1.}}

\author{Eduardo Casas\thanks{Departamento de Matem\'{a}tica Aplicada y Ciencias de la Computaci\'{o}n, E.T.S.I. Industriales y de Telecomunicaci\'on, Universidad de Cantabria, 39005 Santander, Spain, {\tt eduardo.casas@unican.es}.}
\and Daniel Wachsmuth\thanks{Institut f\"ur Mathematik, Universit\"at W\"urzburg, 97074 W\"urzburg, Germany, {\tt  daniel.wachsmuth@mathematik.uni-wuerzburg.de}.}
\and Gerd Wachsmuth\thanks{Technische Universit\"at Chemnitz, Faculty of Mathematics, 09107 Chemnitz, Germany, {\tt gerd.wachsmuth@mathematik.tu-chemnitz.de}.}}

\begin{document}
%%fakesection: Title and abstract
\maketitle
\begin{abstract}
We consider bilinear optimal control problems,
whose objective functionals do not depend on the controls.
Hence, bang-bang solutions will appear.
We investigate sufficient second-order conditions for bang-bang controls, which guarantee
local quadratic growth of the objective functional in $L^1$.
In addition, we prove that for controls that are not bang-bang, no such
growth can be expected.
Finally, we study the finite-element discretization, and
prove error estimates of bang-bang controls in $L^1$-norms.
\end{abstract}



\section{Introduction}

In this article, we consider optimal control problems of the following type:
Minimize the cost functional
\begin{equation}\label{eq001}
 J(y,u):=\frac12 \|y-y_d\|_{L^2(\Omega)}^2
\end{equation}
subject to the elliptic equation
\begin{equation}\label{eq002}
L y + b(y) + \chi_\omega u y = f
\end{equation}
and control constraints
\begin{equation}\label{eq003}
 \alpha \le u \le \beta.
\end{equation}
Here, $\Omega\subset \R^n$ is a bounded domain with Lipschitz boundary,
$L$ is a second-order elliptic operator, and $b$ is a monotone nonlinearity.
The presence of the nonlinear coupling $\chi_\omega u y$ motivates to call this problem `bilinear',
sometimes the term `control affine problem' is used.
In addition, this coupling complicates the analysis considerably.
Since $J$ does not depend explicitly on the control, it is expected that
locally optimal controls $\bar u$ are of bang-bang type, that is $\bar u(x)\in \{\alpha,\beta\}$
for almost all $x \in \Omega$.

We are interested in sufficient second-order optimality conditions and discretization error estimates for problem \eqref{eq001}--\eqref{eq003}.
To this end, we develop an abstract framework in \cref{S2}.
The analysis relies on a structural assumption on the behavior of the reduced gradient
on almost inactive sets.
This allows to prove a second-order condition, see \cref{T2.2}.
The abstract results are then applied in \cref{S3} to the bilinear control problem
of elliptic equations.

In addition, we investigate the discretization of the original problems
using finite elements. Here, we show that under the sufficient second-order
condition we obtain an error estimate of the type
\[
 \| \bar u - \bar u_h \|_{L^1} \le  \ c\, h,
\]
see \cref{T4.2}. This extends earlier result for linear-quadratic bang-bang problems
\cite{Deckelnick-Hinze2010,Wachsmuth2015} and
regularized nonlinear control problems \cite{ACT02,Casas02}.

Let us comment on the existing literature for bang-bang control problems.
The present paper continues our research on bang-bang problems. It extends
earlier works \cite{Casas2012,CWW2017}, which focused on problems with the control
appearing linearly, to the bilinear case.
In the literature on control problems governed by ordinary differential equations
there are many contributions dealing with second-order conditions
in the bang-bang case, e.g.,
\cite{Felgenhauer2003,MaurerOsmolovskii2003,MaurerOsmolovskii2004,MilyutinOsmolovskii1998,Osmolovskii1994,OsmolovskiiMaurer2005,OsmolovskiiMaurer2007}.
In these contributions one typically assumes that the (differentiable) switching function
$\sigma : [0,T] \to \R$ has finitely many zeros.
Our structural assumption \eqref{eq:measure_regularity} can be considered as an extension to the distributed parameter case.

Bilinear control problems for time-dependent equations were studied, e.g., in \cite{AronnaBonnansGoh2016,AronnaBonnansKroner16},
see also the references in these papers. By means of the Goh transform,
the bilinear control problem is transferred into a problem, where the control
appears linearly. It is an open problem, whether the idea of Goh transform can be
applied to control of elliptic (thus time-independent) equations.

\section{Abstract framework}
%%fakesubsection: Intro
\label{S2}
Throughout this section we assume that $(X,\mathcal{B},\eta)$ is a finite and complete measure space. We consider the abstract optimization problem
\begin{equation*}
	\label{eq:prob}
	\tag{\textup{P}}
	\begin{aligned}
		&\text{Minimize }\ J(u) \\
		&\text{subject to }\ u \in \uad,
	\end{aligned}
\end{equation*}
where
\begin{equation}
\uad = \{u \in L^\infty(X) : \alpha \le u(x) \le \beta \quad \text{a.e.~in } X\}
\label{uad}
\end{equation}
with $-\infty < \alpha < \beta < +\infty$, and $J:\uad \to \R$ is a given function.

In the sequel, we will denote the open ball with respect to the $L^p(X)$-norm of radius $r>0$ around $v\in L^p(X)$ by $B^p_r(v)$.

\subsection{A negative result in the non-bang-bang case}
\label{S21}
In this section, we prove that we cannot expect \emph{any} growth of the objective,
if the optimal control is not of bang-bang type.
\begin{theorem}
	\label{thm:no-growth}
	Let us assume that the measure space $(X,\BB,\eta)$
	is additionally separable and non-atomic.
	Suppose that $\bar u$ is a local minimizer of \Pb
	in the sense of $L^1(X)$, which is not bang-bang.
	Further, we assume that $J$ is weak* sequentially continuous from $L^\infty(X)$ to $\R$.
	Then, there exists $\delta_0 > 0$ such that for any $\delta \in (0,\delta_0]$ and
	for any $\varepsilon > 0$,
	there exists $u \in \uad$
	with
	\begin{equation}
		\label{eq:no-growth}
		\norm{u - \bar u}_{L^1(X)} = \delta\ \text{ and }\ J(u) \le J(\bar u) + \varepsilon.
	\end{equation}
\end{theorem}
Before proving the theorem we give some remarks and an auxiliary lemma. First, the theorem implies that
a growth of type
\begin{equation*}
	J(u) \ge J(\bar u) + \nu \norm{ u - \bar u}_{L^p(X)}^\gamma\quad \forall u \in \uad \cap B^1_\delta(\bar u)
\end{equation*}
for some $\nu, \delta, \gamma > 0$ and $p \in [1,\infty]$
is \emph{impossible}. Indeed, let us argue by contradiction. Without loss of generality we can assume that the above growth holds for some $\delta < \delta_0$. Then, according to the theorem, for every $\varepsilon > 0$ there exists $u_\varepsilon \in \uad$ such \eqref{eq:no-growth} holds. This implies with the assumed growth condition and H\"older's inequality that
\begin{align*}
\delta &= \|u_\varepsilon - \bar u\|_{L^1(X)} \le \eta(X)^{1 - \frac{1}{p}}\|u_\varepsilon - \bar u\|_{L^p(X)}\\
& \le \eta(X)^{1 - \frac{1}{p}}\Big(\frac{J(u_\varepsilon) - J(\bar u)}{\nu}\Big)^{1/\gamma} \le \frac{\eta(X)^{1 - \frac{1}{p}}}{\nu^{1/\gamma}}\varepsilon^{1/\gamma}.
\end{align*}
Finally, making $\varepsilon \to 0$ we get a contradiction.

Furthermore, even a growth of type $f(\norm{u-\bar u}_{L^p(X)})$ cannot be satisfied,
as long as $f$ is a non-decreasing function and $f(t) > 0$ for $t > 0$.

Recall that the measure space is non-atomic, if
for all $A \in \BB$ with $\eta(A) > 0$,
there is $B \in \BB$ with $B\subset A$ and $0 < \eta(B) < \eta(A)$.
The measure space is called separable, if there is a countable subset
$\{A_n\} \subset \BB$, such that
\begin{equation*}
	\forall A \in \BB \text{ and } \forall\varepsilon > 0
	\ \exists A_n :
	\eta\bigl( (A \setminus A_n) \cup (A_n \setminus A) \bigr) < \varepsilon
\end{equation*}
holds.
It is easy to check that this is equivalent to the separability of $L^p(X)$ for all $p \in [1,\infty)$. In particular, all regular Borel measures are separable measures.

Before proving the theorem, we need to state a lemma.
\begin{lemma}
	\label{lem:weak_approximation}
	Let the measure space $(X,\BB,\eta)$
	be as in \cref{thm:no-growth}.
	Let a measurable set $B \subset X$ be given.
	Then, there exists a sequence $\{v_k\} \subset L^\infty(X)$
	such that
	$v(x) = 0$ for a.a.\ $x \in X \setminus B$,
	$v(x) \in \{-1,1\}$ for a.a.\ $x \in B$
	and $v \weaklystar 0$ in $L^\infty(X)$.
\end{lemma}
\begin{proof}
	We define the set
	\begin{equation*}
		\F = \{v \in L^2(B) : v(x) \in \{-1,1\} \text{ for a.a.\ } x \in B\}.
	\end{equation*}
	Then, according to \cite[Proposition~6.4.19]{PapageorgiouKyritsi-Yiallourou2009},
	we have
	\begin{equation*}
		\overline{\F}^w = \{v \in L^2(B) : v(x) \in [-1,1] \text{ for a.a.\ } x \in B\}.
	\end{equation*}
	where $\overline{\F}^w$ is the closure of $\F$ w.r.t.\ the weak topology of $L^2(B)$.
	The space $L^2(B)$ is reflexive and separable, since $(X,\BB,\eta)$ is assumed to be separable.
	Hence, the weak topology is metrizable on the bounded set $\overline{\F}^w$.
	Thus, there is a sequence $\{v_k\} \subset L^2(B)$ with
	$v_k \in \F$ and
	$v_k \weakly 0$ in $L^2(B)$.
	Since $\{v_k\}$ is bounded in $L^\infty(B)$, the density of $L^2(B)$ in $L^1(B)$
	implies $v_k \weaklystar 0$ in $L^\infty(B)$.
	Finally, the result follows if $v_k$ is extended by $0$ to $X$.
\end{proof}

Now we are in the position to prove \cref{thm:no-growth}.
\begin{proof}[Proof of \cref{thm:no-growth}]
	Since $\bar u$ is not bang-bang, the set
	$B = \{x \in X : \alpha + \rho \le \bar u \le \beta - \rho\}$
	has positive measure for some $\rho > 0$.
	We apply \cref{lem:weak_approximation} and obtain a sequence $\{v_k\} \subset L^\infty(X)$
	with the properties stated in \cref{lem:weak_approximation}. Set $\delta_0 = \rho \eta(B)$. Then,
	given $\delta \le \delta_0$, we consider the controls
	$u_k = \bar u + \frac{\delta}{\eta(B)} v_k$
	and obtain $u_k \in \uad$.
	Moreover, we have $\norm{u_k - \bar u}_{L^1(X)} = \delta$
	for all $k$.
	The weak* sequential continuity of $J$ implies
	$J(u_k) \to J(\bar u)$.
	Thus, for any $\varepsilon > 0$ there exists $k_\varepsilon \ge 1$ such that $J(u_k) - J(\bar u) < \varepsilon$ $\forall k \ge k_\varepsilon$, which implies \eqref{eq:no-growth}.
\end{proof}

\subsection{Second-order analysis}
\label{S22}
In this section, we consider the second-order analysis of problem \Pb.
To this end,
let $\bar u \in \uad$ be a fixed control. We make the following assumptions on $J$ and $\bar u$.
\begin{enumerate}[label=\textup{\bfseries(H\arabic*)}]
	\item\label{H1}
		The functional $J$ can be extended to an $L^\infty(X)$-neighborhood $\mathcal{A}$ of $\uad$. It is twice continuously Fr\'echet differentiable w.r.t.\ $L^\infty(X)$
		in this neighborhood. Moreover, we assume that $\bar u$ satisfies the first-order condition $J'(\bar u) (u - \bar u) \ge 0$ for all $u \in \uad$.
	\item\label{H2}
		The second derivative $J''(\bar u) : L^\infty(X)^2 \to \R$
		can be extended continuously to $L^q(X)^2$, for some $q \in [1,3/2)$. In particular, there is a constant $C > 0$, such that
		\begin{equation}
			\label{eq:boundedness_second_deriv}
				\abs{J''(\bar u) (v_1 , v_2)}
				\le
				C \, \norm{v_1}_{L^q(X)}
				\, \norm{v_2}_{L^q(X)}
				\\
		\end{equation}
		holds for all $v_1,v_2 \in L^q(X)$.
	\item\label{H3}
		For each $\varepsilon > 0$ there is  $\delta_\varepsilon > 0$ such that
		\begin{equation}
			\label{eq:contin_second_deriv_alt}
			\bigabs{[J''(u_\theta) - J''(\bar u)] (u-\bar u)^2}
			\le
			\varepsilon \, \norm{u-\bar u}_{L^1(X)}^2
		\end{equation}
		holds for all $u \in \uad \cap B_{\delta_\varepsilon}^1(\bar u)$, $u_\theta = \bar u + \theta(u-\bar u)$ and any $0 \le \theta \le 1$.
	\item\label{H4}
		There exists a function $\bar\psi \in L^1(X)$,
		such that $J'(\bar u) \, v = \int_X \bar\psi \, v \, \deta$
		for all $v \in L^\infty(X)$.
	\item\label{H5}
		There exists a constant $K > 0$, such that
		\begin{equation}
			\label{eq:measure_regularity}
			\eta(\{x \in X : \abs{\bar\psi(x)} \le \varepsilon\}) \le K \, \varepsilon
		\end{equation}
		is satisfied for all $\varepsilon > 0$.
\end{enumerate}

Let us observe that \ref{H1}, \ref{H4} and \ref{H5} imply that $\bar u$ is a bang-bang control.

Under the previous assumptions we can prove some sufficient second-order optimality conditions for $\bar u$. To this end we introduce the following cone of critical directions: for every $\tau > 0$ we define
\begin{equation}
	\label{eq:critical_cone}
	C_{\bar u}^\tau := \big\{v \in L^2(X) : v(x) = 0 \text{ if } |\bar\psi(x)| > \tau \text{ and } v \text{ satisfies \eqref{E2.13}}\big\}
\end{equation}
with
\begin{equation}
	v(x)
	\begin{cases}
		\ge 0 &\text{ if } \bar{u}(x) = \alpha,\\
		\le 0 &\text{ if } \bar{u}(x) = \beta,
	\end{cases}
	\quad\text{for a.a.\ } x \in X.
\label{E2.13}
\end{equation}

Before establishing the second-order conditions we state the following result, whose proof can be found in \cite[Proposition 2.7]{CWW2017}.

\begin{theorem}\label{T2.1}
Let us assume that \ref{H1}, \ref{H4} and \ref{H5} hold, then
\begin{equation}
J'(\bar u)(u - \bar u) \ge \kappa \|u - \bar u\|^2_{L^1(X)}\quad \forall u \in \uad,
\label{E2.24}
\end{equation}
where $\kappa = (4(\beta - \alpha)K)^{-1}$.
\end{theorem}

The next theorem provides a second-order condition
which allows us to prove a quadratic growth of the objective $J$
in the neighborhood of $\bar u$.
In particular, $\bar u$ is a strict local solution
under this assumption.
Note that condition \eqref{eq305}
is slightly weaker
than the corresponding results
\cite[Theorems~2.8 and 3.3]{CWW2017},
which required $\kappa' < \kappa$ in \eqref{eq305}.
This improvement has been possible by some slightly more refined
estimates in the proof.
\begin{theorem}\label{T2.2}
Suppose that the above assumptions \ref{H1}--\ref{H5} are satisfied. Let $\kappa$ be as in \cref{T2.1}. Further, we assume that
\begin{equation}\label{eq305}
		\exists \tau > 0,\; \exists \kappa' < 2\kappa :\quad
		J''(\bar u)v^2 \ge -\kappa'\|v\|^2_{L^1(X)} \ \ \forall v \in C^\tau_{\bar u}.
	\end{equation}
	Then, there exist $\nu > 0$ and $\delta > 0$ such that
	\begin{equation}
		J(\bar u) + \nu\|u - \bar u\|^2_{L^1(X)} \le J(u) \ \ \ \forall u \in \uad \cap B^1_\delta(\bar u).
		\label{E2.9}
\end{equation}
\end{theorem}

The following lemma will be used to prove this theorem.

\begin{lemma}\label{L2.1}
Suppose that the above assumptions \ref{H1}--\ref{H5} are satisfied. Let $\kappa$ be as in \cref{T2.1}. Further, we assume that
there exist $\tau > 0$ and $\kappa' \ge 0$
such that
\begin{equation}\label{eq305var}
	J''(\bar u)v^2 \ge -\kappa'\|v\|^2_{L^1(X)} \ \ \forall v \in C^\tau_{\bar u}.
\end{equation}
Then,
for every $\gamma \in (0,3\kappa)$,
there is a $\delta > 0$
such that
\begin{equation}
	J'(\bar u)(u - \bar u) + J''(u_\theta)(u - \bar u)^2 \ge (\kappa - \kappa' - \gamma)\|u - \bar u\|^2_{L^1(X)} \ \ \forall u \in \uad \cap B^1_\delta(\bar u),
\label{E2.25}
\end{equation}
where $u_\theta = \bar u + \theta(u - \bar u)$ and $0 \le \theta \le 1$ is arbitrary.
\end{lemma}
\begin{proof}
	We follow the idea of the proofs of \cite[Theorems~2.8 and 3.3]{CWW2017}.
	First, we note that \eqref{eq:boundedness_second_deriv}
	implies that
	\begin{align}
		\abs{J''(\bar u) (v_1,v_2)}
		&\le
		C \, \norm{v_1}_{L^1(X)}^{1/q} \, \norm{v_2}_{L^1(X)}^{1/q} \, \norm{v_1}_{L^\infty(X)}^{(q-1)/q} \, \norm{v_2}_{L^\infty(X)}^{(q-1)/q}
		\label{eq:bounds_Jdd}
	\end{align}
	holds $\forall v_1, v_2 \in L^\infty(X)$.
	Now, let $u \in \uad$ with $\norm{u - \bar u}_{L^1(X)} \le \delta$ be given,
	where $\delta > 0$ will be specified later.
	We define
	\[
		u_1(x) :=
		\begin{cases}
			\bar u(x) & \text{if } x \in X_\tau,\\
			u(x) & \text{otherwise,}
		\end{cases}
		\quad\text{and}\quad
		u_2(x)  :=
		\begin{cases}
			u(x) - \bar u(x) & \text{if } x \in X_\tau,\\
			0 & \text{otherwise,}
		\end{cases}
	\]
	where $X_\tau = \{x \in X : |\bar\psi(x)| > \tau\}$.
	Then we have that $u = u_1 + u_2$, $(u_1 - \bar u) \in C_{\bar u}^\tau$,
	and $|u_1 - \bar u| \le |u - \bar u|$ a.e.\ in $X$.
	Let $\gamma \in (0, 3\kappa)$ be given.
	Now, we can use
	\eqref{eq305var}, \eqref{eq:bounds_Jdd}
	and Young's inequality
	to obtain for generic positive constants $C$
	\begin{align*}
		J''(\bar u) (u - \bar u)^2
		&=
		J''(\bar u) (u_1 - \bar u)^2
		+
		2 \, J''(\bar u) (u_1 - \bar u, u_2)
		+
		J''(\bar u) u_2^2
		\\
		&\ge
		-\kappa' \, \norm{u_1 - \bar u}_{L^1(X)}^2
		- C \, \norm{u_1 - \bar u}_{L^1(X)}^{1/q} \, \norm{u_2}_{L^1(X)}^{1/q}
		- C \, \norm{u_2}_{L^1(X)}^{2/q}
		\\
		&\ge
		-\Big(\kappa' + \frac\gamma3\Big) \, \norm{u_1 - \bar u}_{L^1(X)}^2
		- C \, \norm{u_2}_{L^1(X)}^{2/(2q-1)}
		- C \, \norm{u_2}_{L^1(X)}^{2/q}.
	\end{align*}
	Owing to the construction of $u_1$ and $u_2$, we have
	for $\delta$ small enough
	\begin{equation}
		\label{eq:in_proof_of_some_lemma}
		J''(\bar u)(u - \bar u)^2
		\ge
		-\Big(\kappa'+\frac\gamma3\Big) \|u-\bar u\|_{L^1(X)}^2 - C \|u - \bar u\|_{L^1(X_\tau)}^{\bar q}
	\end{equation}
	with $\bar q = \min(2/(2q-1), \,2/q) = 2/(2q-1) > 1$,
	since $1 \le q < 3/2$.
	Next, we use \cref{T2.1} to infer
	\begin{align}
		\notag
		J'(\bar u) (u - \bar u)
		&=
		\Big(1 - \frac\gamma{3\,\kappa} \Big) \, J'(\bar u) (u - \bar u)
		+
		\frac\gamma{3\,\kappa} \, J'(\bar u) (u - \bar u)
		\\
		\notag
		&\ge
		\Big(\kappa - \frac\gamma{3}\Big) \, \norm{u - \bar u}_{L^1(X)}^2
		+
		\frac\gamma{3\,\kappa} \, \int_{X_\tau} \abs{\bar\psi} \, \abs{u - \bar u} \, \deta
		\\
		\label{eq:first_order_growth_in_lemma}
		&\ge
		\Big(\kappa - \frac\gamma{3}\Big) \, \norm{u - \bar u}_{L^1(X)}^2
		+
		\frac{\gamma\,\tau}{3\,\kappa} \norm{u - \bar u}_{L^1(X_\tau)}
		.
	\end{align}
	Furthermore, assumption \ref{H3} implies
	\begin{equation}
		\label{eq:remainder_in_lemma}
		\bigabs{[J''(u_\theta) - J''(\bar u)] \, (u-\bar u)^2}
		\le
		\frac\gamma3 \, \norm{u-\bar u}_{L^1(X)}^2
	\end{equation}
	if $\delta$ is chosen small enough.
	Now, by adding the inequalities
	\eqref{eq:in_proof_of_some_lemma}, \eqref{eq:first_order_growth_in_lemma}
	and \eqref{eq:remainder_in_lemma}, we have
	\begin{align*}
		J'(\bar u ) (u - \bar u)
		+
		J''(u_\theta) (u - \bar u)^2
		&\ge
		(\kappa - \kappa' - \gamma) \, \norm{u - \bar u}_{L^1(X)}^2
		\\
		&\qquad
		+
		\frac{\gamma \, \tau}{3 \, \kappa} \, \norm{u - \bar u}_{L^1(X_\tau)}
		- C \|u - \bar u\|_{L^1(X_\tau)}^{\bar q}
		.
	\end{align*}
	Note that the sum of the terms on the second line
	is non-negative if $\delta$ is small enough,
	since $\bar q > 1$.
\end{proof}

Now we are in the position to prove \cref{T2.2}.
\begin{proof}[Proof of \cref{T2.2}]
	Let $\tau > 0$ and $\kappa' < 2 \, \kappa$ be given,
	such that \eqref{eq305} is satisfied.
	Without loss of generality, we assume that $\kappa' \ge 0$.
	We choose $\gamma \in (0, 2 \, \kappa - \kappa')$.
	We apply \cref{L2.1}
	and get $\delta > 0$ such that
	\eqref{E2.25} holds.
	Now, we choose an arbitrary $u \in \uad \cap B^1_\delta(\bar u)$.
	Using a Taylor expansion,
	we get
	\begin{equation*}
		J(u) - J(\bar u)
		=
		J'(\bar u) (u - \bar u)
		+
		\frac12 J''(u_\theta) (u - \bar u)^2
	\end{equation*}
	for $u_\theta = \bar u + \theta (u - \bar u)$ and $0 \le \theta \le 1$.
	Now, we apply \eqref{E2.25}
	and \cref{T2.1} to conclude
	\begin{align*}
		J(u) - J(\bar u)
		&=
		\frac12J'(\bar u) (u - \bar u)
		+
		\frac12J'(\bar u) (u - \bar u)
		+
		\frac12 J''(u_\theta) (u - \bar u)^2
		\\
		&\ge
		\frac\kappa2 \, \norm{u - \bar u}_{L^1(X)}^2
		+
		\frac12 (\kappa - \kappa' - \gamma) \, \norm{u - \bar u}_{L^1(X)}^2
		\\
		&\ge
		\frac12 (2 \, \kappa - \kappa' - \gamma) \, \norm{u - \bar u}_{L^1(X)}^2
		.
	\end{align*}
	Since $\nu := (2 \, \kappa - \kappa' - \gamma)/2 > 0$,
	the assertion follows.
\end{proof}

\subsection{Approximation results}
\label{S23}
The rest of this section is dedicated to the numerical approximation of the optimization problem \Pb. To this end we make the following assumptions.
First, we fix an approximation of the underlying set $X$.
\begin{enumerate}[label=\textup{\bfseries(D\arabic*)}]
	\item\label{D1}
		There is a sequence of measurable subsets $X_h \subset X$,
		such that $\eta(X \setminus X_h) \to 0$ as $h \to 0$.
\end{enumerate}
Associated with the approximation $X_h$ of $X$, we define the following two notions of convergence.
For a sequence $u_h \in L^1(X_h)$ and $u \in L^1(X)$,
we say that $u_h \to u$ in $L^1(X)$ if and only if $\norm{u_h - u}_{L^1(X_h)} \to 0$ as $h \to 0$.
Similarly, for a sequence $u_h \in L^\infty(X_h)$ and $u \in L^\infty(X)$,
we say that $u_h \weaklystar u$ in $L^\infty(X)$ if and only if
$\int_{X_h} v \, u_h \, \deta \to \int_X v \, u \, \deta$ as $h \to 0$ for all $v\in L^1(X)$.
Due to $\eta(X_h \setminus X) \to 0$, both notions of convergence are equivalent to
$(u_h + f\,\chi_{X\setminus X_h}) \to u$ in $L^1(X)$
and
$(u_h + f\,\chi_{X\setminus X_h}) \weaklystar u$ in $L^\infty(X)$, respectively,
where $f \in L^\infty(X)$ is an arbitrary, but fixed extension of $u_h$.

Next, we state assumptions to define the approximation of our problem \Pb.
\begin{enumerate}[label=\textup{\bfseries(D\arabic*)},resume]
	\item\label{D1p5} The sets $\uadh \subset L^\infty(X_h)$ are closed, convex and contained in
		the set $\{u_h \in L^\infty(X_h) : \alpha \le u_h \le \beta \text{ a.e.\ in } X_h\}$.
		Moreover, for every $u \in \uad$ there exists a sequence $u_h \in \uadh$ such that $u_h \to u$ in $L^1(X)$ as $h \to 0$.
	\item\label{D2} $\{J_h\}_h$ is a sequence of functions $J_h : \uadh \longrightarrow \R$ that are weakly lower semicontinuous with respect to the $L^2(X)$ topology.

	\item\label{D3} The following properties hold for sequences $u_h \in \uadh$ and $u \in \uad$
\begin{align}
&\text{If } u_h \stackrel{*}{\rightharpoonup} u \ \text{ in } \ L^\infty(X),\ \text{ then } \ J(u) \le \liminf_{h \to 0}J_h(u_h),\label{D3.1}\\
&\text{If } u_h \to u \ \text{ in } \ L^1(X),\ \text{ then } \ J(u) = \lim_{h \to 0}J_h(u_h).\label{D3.2}
\end{align}

	\item\label{D4} The functions $J_h$ have $C^1$ extensions $J_h : \mathcal{A}_h \longrightarrow \R$,
		where $\AA_h \subset L^\infty(X_h)$ is a neighborhood of $\uadh$.
		Moreover, for all $u_h \in \uadh$ and for all $u \in \uad$, $J_h'(u_h)$ and $J'(u)$ are linear and continuous forms on $L^1(X_h)$ and $L^1(X)$, respectively.
		Hence, there exist elements $\psi_h \in L^\infty(X_h)$, $\psi \in L^\infty(X)$
		such that the following identifications hold: $J_h'(u_h) = \psi_h$ and $J'(u) = \psi$.
\end{enumerate}

Now, we define the approximating problems
\begin{equation*}
	\label{eq:probh}
	\tag{\textup{P}$_h$}
	\begin{aligned}
		&\text{Minimize }\  J_h(u_h) \\
		&\text{subject to }\  u_h \in \uadh.
	\end{aligned}
\end{equation*}

First, we state a lemma which provides a partial converse to \ref{D1p5}.
\begin{lemma}
	\label{lem:weak_uadh}
	Let us assume that \ref{D1} and \ref{D1p5} hold.
	Let $u_h \subset \uadh$ be a sequence with $u_h \weaklystar u$ in $L^\infty(X)$
	for some $u \in L^\infty(X)$.
	Then, $u \in \uad$ holds.
	If, additionally, $\norm{u_h - \bar u}_{L^1(X_h)} \le \delta$
	for some $\bar u \in \uad$ and some $\delta > 0$, and for all $h > 0$,
	we get $\norm{u - \bar u}_{L^1(X)} \le \delta$.
\end{lemma}
\begin{proof}
	We argue by contradiction.
	Assume that $u \le \beta$ is not satisfied a.e.\ on $X$.
	Then, there is a measurable set $B \subset X$ with $\eta(B) > 0$
	and $\varepsilon > 0$ such that
	$u \ge \beta + \varepsilon$ a.e.\ in $B$.
	If $h$ is small enough, we have $\eta(X \setminus X_h) < \eta(B)/2$, hence $\eta(B \cap X_h) > \eta(B)/2$.
	Together with $u_h \le \beta$, this implies
	\begin{equation*}
		\int_{X_h} \chi_B \, (u_h - u) \, \deta
		=
		\int_{B \cap X_h} (u_h - u) \, \deta
		\le
		\int_{B \cap X_h} [\beta - (\beta + \varepsilon)] \, \deta
		\le
		-\frac12 \, \eta(B) \, \varepsilon,
	\end{equation*}
	which contradicts $u_h \weaklystar u$ in $L^\infty(X)$.
	Similar arguments can be used if $u \ge \alpha$ is violated.

	It remains to check the second assertion.
	By extending $u_h$ with $\bar u$ on $X \setminus X_h$,
	we get $u_h \weaklystar u$ in $L^\infty(X)$,
	in particular, $u_h \weakly u$ in $L^1(X)$.
	Now, the assertion follows from the weak lower semicontinuity of the norm of $L^1(X)$.
\end{proof}

The following theorem proves that \Pbh realizes a convergent approximation of \Pb.
\begin{theorem}\label{T2.3}
Let us assume that \ref{D1}--\ref{D3} hold. Then for every $h$, the problem \Pbh has at least a global solution $\bar u_h$.
Furthermore, if $\{\bar u_h\}_h$ is a sequence of global solutions of \Pbh, and
$\bar u_h \weaklystar \tilde u$ in $L^\infty(X)$
then $\tilde u$ is a global solution of \Pb.
Conversely, if $\bar u$ is a bang-bang strict local minimum of \Pb in the $L^1(X)$ sense, then there exists a sequence $\{\bar u_h\}_h$ of
local minimizers of problems \Pbh in the sense of $L^1(X_h)$ such that $\bar u_h \to \bar u$ in $L^1(X)$.
\end{theorem}

\begin{proof}
The existence of a global solution $\bar u_h$ of \Pbh follows from the
boundedness, convexity and closedness of $\uadh$ and the weak lower
semicontinuity of $J_h$; see assumptions \ref{D1p5} and \ref{D2}. Now, consider a
subsequence, denoted in the same way, such that $\bar u_h
\stackrel{*}{\rightharpoonup} \tilde u$ in $L^\infty(X)$. Since $\bar u_h \in \uadh$ for every $h$,
the inclusion $\tilde u \in \uad$ holds by \cref{lem:weak_uadh}. Furthermore, given an
element $u \in \uad$, according to assumption \ref{D1p5} we can take a sequence
$\{u_h\}_h$ with $u_h \in \uadh$ such that $u_h \to u$ in $L^1(X)$. Then, using
\ref{D3} and the global optimality of every $\bar u_h$, we infer
\[
J(\tilde u) \le \liminf_{h \to 0}J_h(\bar u_h) \le \limsup_{h \to 0}J_h(\bar u_h) \le \limsup_{h \to 0}J_h(u_h) = J(u).
\]
Hence, $\tilde u$ is a solution of \Pb.

Conversely, we assume that $\bar u$ is a bang-bang strict local minimum of \Pb. Then, there exists $\delta > 0$ such that
\[
J(\bar u) < J(u)\quad \forall u \in \uad \cap B^1_\delta(\bar u)\ \text{ with }\ \bar u \neq u.
\]
Then, we consider the problems
\begin{equation*}
	\label{eq:probdh}
	\tag{\textup{P}$_{\delta,h}$}
	\begin{aligned}&\text{Minimize }\  J_h(u_h) \\
		&\text{subject to }\  u_h \in \uadh \text{ and } \norm{u_h - \bar u}_{L^1(X_h)} \le \delta.
	\end{aligned}
\end{equation*}
From \ref{D1p5} we deduce the existence of a sequence $\{u_h\}_h$ with $u_h \in
\uadh$ such that $u_h\to \bar u$ strongly in $L^1(X)$. Hence, for every $h$
small enough we have that $u_h \in \uadh \cap B^1_\delta(\bar u)$. Therefore
the feasible set of \Pbdh is not empty for every $h$ small enough, and
arguing as before we have that \Pbdh has a solution $\bar u_h$ for every $h$
small enough.
Moreover, the sequence $\{\bar u_h\}$ is bounded in $L^\infty(X)$. Thus,
there exists a weak* converging subsequence.
Additionally, for any subsequence converging to $\tilde u$ in
$L^\infty(X)$ weak*, we get that $\tilde u \in \uad \cap
B^1_\delta(\bar u)$ by \cref{lem:weak_uadh}, and as above $J(\tilde u) \le J(\bar u)$.
The strict local
optimality of $\bar u$ in $\uad \cap B^1_\delta(\bar u)$ implies that $\tilde u
= \bar u$. Moreover, we conclude that the whole sequence $\{\bar u_h\}_h$
converges to $\bar u$ in $L^\infty(X)$ weak*. In addition, by using the
bang-bang property of $\bar u$, we get
\[
\|\bar u_h - \bar u\|_{L^1(X_h)} = \int_{\{x \in X_h : \bar u(x) = \alpha\}}(\bar u_h - \bar u)\, \deta + \int_{\{x \in X_h : \bar u(x) = \beta\}}(\bar u - \bar u_h)\, \deta \to 0 \text{ as } h \to 0.
\]
From here we get that $\|\bar u_h - \bar u\|_{L^1(X_h)} < \delta$ for all $h$ small enough. Hence, $\bar u_h$ is a local minimum of \Pbh for every small $h$.
\end{proof}

We finish this section by proving an estimate of $\bar u_h - \bar u$ in terms of the order of the approximations of $\bar u$ by elements of $\uadh$ and $J'$ by $J'_h$.

\begin{theorem}\label{T2.4}
Let us assume that \ref{H1}--\ref{H5} and \ref{D1}--\ref{D4} hold. Additionally, we suppose that $\bar u$ satisfies the second-order condition \eqref{eq305} with $\kappa' \in (0, \kappa)$. Let $\{\bar u_h\}_h$ be a sequence of local solutions of problems \Pbh converging to $\bar u$ in $L^1(X)
$. Then, for $\gamma = (\kappa - \kappa')/2$
we obtain that the estimate
\begin{align}
\|\bar u_h - \bar u\|^2_{L^1(X_h)} &\le \frac{\gamma + 1}{\gamma^2}\|J'_h(\bar u_h) - J'(\hatbar u_h)\|_{L^\infty(X_h)}^2\notag\\
&\qquad+ \frac{1}{\gamma}\inf_{u_h \in \uadh}\Big(\|u_h - \bar u\|^2_{L^1(X_h)} + 2J'(\hatbar u_h)(\hat u_h - \bar u) \Big)
\label{E2.30}
\end{align}
holds for all $h$ small enough, where $\hatbar u_h$ and $\hat u_h$ denote the extensions of $\bar u_h$ and $u_h$ by $\bar u$ to $X$, respectively.
\end{theorem}
This specific extension of  the elements $u_h$ is quite convenient for the derivation of the error estimate.
We will also see in \cref{S4} below, that this will not impede the applicability
of our abstract framework to derive discretization error estimates for optimal control problems.
Let us observe that for every $u_h \in \uadh$, its extension $\hat u$ to $X$ by setting $\hat u(x) = \bar u(x)$ in $X \setminus X_h$ belongs to $\uad$, hence $\hat u \in \AA$ as well.

\begin{proof}
Let $u_h \in \uadh$, and denote by $\hat u_h$ its extension to $X$ by $\bar u$.
Since $\bar u_h$ is a local minimum of \Pbh, $J'_h(\bar u_h)(u_h - \bar u_h) \ge 0$.
Due to \ref{D4} this inequality can be written in the form
\begin{equation}
	\label{eq:in_proof_of_abstract_error}
J'(\hatbar u_h)(\hatbar u_h - \bar u)
\le [J'_h(\bar u_h) - J'(\hatbar u_h)](\chi_{X_h}(\hat u_h - \hatbar u_h)) + J'(\hatbar u_h)(\hat u_h - \bar u).
\end{equation}
Note that our choice of extension is crucial for the above rearrangement.
Next, we rewrite the left-hand side,
and by the mean value theorem and by denoting $u_\theta = \bar u + \theta_h(\hatbar u_h - \bar u)$ with $0 \le \theta_h \le 1$, we infer
\begin{align*}
J'(\hatbar u_h)(\hatbar u_h - \bar u)
&=
J'(\bar u)(\hatbar u_h - \bar u) + [J'(\hatbar u_h) - J'(\bar u)](\hatbar u_h - \bar u).
\\
&=
J'(\bar u)(\hatbar u_h - \bar u) + J''(u_\theta)(\hatbar u_h - \bar u)^2
.
\end{align*}
Taking $\gamma = (\kappa - \kappa')/2$ in \cref{L2.1}, we get for $h$ small enough
\[
\gamma\|\bar u_h - \bar u\|^2_{L^1(X_h)}
=
\gamma\|\hatbar u_h - \bar u\|^2_{L^1(X)}
\le
J'(\hatbar u_h)(\hatbar u_h - \bar u).
\]
This estimate is now used in \eqref{eq:in_proof_of_abstract_error}.
After applying Young's inequality we obtain
\begin{align*}
&\gamma\|\bar u_h - \bar u\|^2_{L^1(X_h)} \le \|J'_h(\bar u_h) - J'(\hatbar u_h)\|_{L^\infty(X_h)}\|u_h - \bar u_h\|_{L^1(X_h)} + J'(\hatbar u_h)(\hat u_h - \bar u)\\
&\quad\le \|J'_h(\bar u_h) - J'(\hatbar u_h)\|_{L^\infty(X_h)}\big(\|u_h - \bar u\|_{L^1(X_h)} + \|\bar u - \bar u_h\|_{L^1(X_h)}\big)
\\&\qquad
+ J'(\hatbar u_h)(\hat u_h - \bar u)\\
&\quad\le \big(\frac{1}{2} + \frac{1}{2\gamma}\big)\|J'_h(\bar u_h) - J'(\hatbar u_h)\|_{L^\infty(X_h)}^2
+ \frac{1}{2}\|u_h - \bar u\|^2_{L^1(X_h)} + \frac{\gamma}{2}\|\bar u_h - \bar u\|^2_{L^1(X_h)}
\\&\qquad
+ J'(\hatbar u_h)(\hat u_h - \bar u).
\end{align*}
From this inequality we deduce
\begin{align*}
\|\bar u_h - \bar u\|^2_{L^1(X_h)} &\le \frac{\gamma + 1}{\gamma^2}\|J'_h(\bar u_h) - J'(\hatbar u_h)\|_{L^\infty(X_h)}^2 \\
&\qquad + \frac{1}{\gamma}\|u_h - \bar u\|^2_{L^1(X_h)} + \frac{2}{\gamma}J'(\hatbar u_h)(\hat u_h - \bar u).
\end{align*}
Since $u_h$ is an arbitrary element of $\uadh$, this inequality implies \eqref{E2.30}.
\end{proof}

In \cref{S4} we will provide precise estimates for the right hand side of \eqref{E2.30} for some distributed optimal control problems, including bilinear controls.

\section{Second-order analysis for bilinear control problems}
\label{S3}

In this section, we apply the second-order analysis results proved in the abstract framework in \cref{S2} to the study of some optimal control problems. The first part of this section will be devoted to the analysis of a bilinear distributed control problem associated with a semilinear elliptic equation. In the second part, we will consider a bilinear Neumann control problem.

In what follows, $\Omega$ denotes a bounded open subset of $\R^n$, $1 \le n \le 3$, with a Lipschitz boundary $\Gamma$. In $\Omega$ we consider the elliptic partial differential operator
\begin{equation}
Ay = -\sum_{i, j = 1}^n\partial_{x_j}[a_{ij}\partial_{x_i}y] + a_0y,
\label{E3.0}
\end{equation}
where $a_{ij}, a_0 \in L^\infty(\Omega)$ and $a_0 \ge 0$ in $\Omega$. Associated with this operator we define the usual bilinear form $a:H^1(\Omega) \times H^1(\Omega) \longrightarrow \R$
\begin{equation}
a(y,z) = \int_\Omega\Big(\sum_{i, j= 1}^na_{ij}(x)\partial_{x_i}y(x)\partial_{x_j}z(x) + a_0(x)y(x)z(x)\Big)\, \dx.
\label{E3.a}
\end{equation}
Let $\Gamma_D$ be a closed subset of $\Gamma$, possibly empty, and set $\Gamma_N = \Gamma \setminus \Gamma_D$. We define the space
\[
V = \{y \in H^1(\Omega) : y = 0 \text{ on } \Gamma_D\}.
\]
equipped with the usual norm of $H^1(\Omega)$
and the operator $L:V \longrightarrow V^*$ via
\[
\langle Ly,z \rangle = a(y,z)\quad \forall y, z \in V,
\]
and we assume its coercivity.
\begin{enumerate}[label=\textup{\bfseries(A\arabic*)},wide=0pt]
\item\label{asm:coercivity}
We have that
\begin{equation}
\exists\Lambda > 0 \text{ such that } \Lambda\|y\|^2_V  \le a(y,y)\quad \forall y \in V.
\label{E3.1}
\end{equation}
\end{enumerate}

Moreover, we consider a Carath\'eodory function $b:\Omega \times \R \longrightarrow \R$ of class $C^2$ with respect to the second variable, such that the following assumptions are satisfied.

\begin{enumerate}[label=\textup{\bfseries(A\arabic*)},wide=0pt,resume]
	\item\label{asm:1}
		We assume that $b(\cdot,0) = 0$,
		\[
			\frac{\partial b}{\partial y}(x,y) \geq 0 \quad  \mbox{ for a.a. } x \in \Omega \text{ and for all } y \in \R,
		\]
		and that for all $M>0$ there exists a constant $C_{b,M}>0$ such that
		the boundedness estimate
		\[
			\left|\frac{\partial b}{\partial y}(x,y)\right|+
			\left|\frac{\partial^2 b}{\partial y^2}(x,y)\right| \leq C_{b,M}
			\mbox{ for a.e. } x \in \Omega \mbox{ and for all } |y| \leq M,
		\]
		and that for all $\varepsilon > 0$ and $M > 0$ there exists $\rho_{\varepsilon,M} > 0$ such that for a.e.~$x \in \Omega$
		\[
			\left|\frac{\partial^2b}{\partial y^2}(x,y_2) -
			\frac{\partial^2b}{\partial y^2}(x,y_1)\right| < \varepsilon
		 \mbox{ and for all } \abs{y_1}, \abs{y_2} \leq M \text{ with } |y_2 - y_1| < \rho_{\varepsilon,M}
		\]
		are satisfied. In what follows we use the notation
		\[
		b' = \frac{\partial b}{\partial y} \ \text{ and }\ b'' = \frac{\partial^2b}{\partial y^2}.
		\]
\end{enumerate}

\subsection{A bilinear distributed control problem}
\label{S31}

In this section, we consider the following state equation
\begin{equation}
L y + b(\cdot,y) + \chi_\omega u y = f \ \text{ in } V^*,
\label{E3.2}
\end{equation}
where $\omega$ is an open subset of $\Omega$, and $u$ and $f$ satisfy the following assumptions.
\begin{enumerate}[label=\textup{\bfseries(A\arabic*)},wide=0pt,resume]
	\item\label{asm:3}
	We fix $\bar p > n$ and $\bar p' = \bar p/(\bar p - 1)$ is its conjugate. We assume that $f \in W^{1,\bar p'}(\Omega)^*$.
	\item\label{asm:4}
	We assume that $u \in \mathcal{A}$, where the open set $\mathcal{A} \subset L^\infty(\omega)$ is given by
	\[
\mathcal{A} = \{v \in L^\infty(\omega) : \exists\varepsilon_v > 0 \text{ such that } v(x) > -\frac{\Lambda	}{2} + \varepsilon_v\ \text{ for a.a. } x \in \omega\},
\]
where $\Lambda$ was introduced in \ref{asm:coercivity}.
\end{enumerate}
In the next theorem, we analyze the equation \eqref{E3.2}.

\begin{theorem}
	\label{T3.1}
	The following statements hold.
	\begin{enumerate}[label=(\arabic*)]
		\item For any $u \in \AA$
			there exists a unique solution $y_u \in Y := V \cap L^\infty(\Omega)$ of the state equation \eqref{E3.2}.
			Moreover, there exists a constant $C$ such that
			\begin{equation}
			\|y_u\|_{Y} =	\|y_u\|_{L^\infty(\Omega)} + \|y_u\|_{V} \le C \ \ \forall u \in \AA.
				\label{E3.3}
			\end{equation}
		\item
			The control-to-state mapping $G:\AA \longrightarrow Y$
			defined by $G(u) = y_u$ is of class $C^2$.
			Moreover, for $v \in L^\infty(\Omega)$, $z_v = G'(u)\,v$ is the unique solution of
			\begin{equation}
				L \, z_v + b'(\cdot, y_u) \, z_v + \chi_\omega u \, z_v + y_u \,\chi_\omega v = 0
				\label{E3.4}, % not cited
			\end{equation}
			and given $v_1, v_2 \in L^2(\Omega)$, $w_{v_1,v_2} =
			G''(u)(v_1,v_2)$ is the unique solution of
			\begin{equation}
				\begin{aligned}
				&L \, w_{v_1,v_2}
				+ b'(\cdot, y_u) \, w_{v_1,v_2}
				+ \chi_\omega u \, w_{v_1,v_2}
				\\
				&\qquad\qquad
				+ b''(\cdot, y_u) \, z_{v_1} \, z_{v_2}
				+ \chi_\omega v_1 \, z_{v_2}
				+ \chi_\omega v_2 \, z_{v_1}
				= 0
				\end{aligned}
				\label{E3.5} % not cited
			\end{equation}
			where $z_{v_i} = G'(u)\,v_i$, $i = 1, 2$.
	\end{enumerate}
\end{theorem}

\begin{proof}
For the proof of existence and uniqueness of a solution of \eqref{E3.2} in $Y$, first we observe that the linear operator $L + \chi_\omega u$ is coercive in $V$ for all $u \in \AA$ due to the fact that $u \ge -\frac{\Lambda}{2}$ and assumption \ref{asm:coercivity}. Then, the arguments are standard; see, for instance, \cite[\S 4.1]{Troltzsch2010}. We recall that the boundedness of $y$ needed in this proof is a consequence of Stampacchia's result \cite[Theorem 4.2]{Stampacchia65}. To prove the differentiability of the mapping $G$ we use the implicit function theorem as follows. We define
\[
Y_{\bar p} =\{y \in Y : Ly \in W^{1,\bar p'}(\Omega)^*\},
\]
which is a Banach space when it is endowed with the graph norm. Now, we consider the mapping $\mathcal{L}:Y_{\bar p} \times \AA \longrightarrow W^{1,\bar p'}(\Omega)^*$ given by
\[
\mathcal{L}(y,u) = Ly + b(\cdot,y) + \chi_\omega uy - f.
\]
From assumption \ref{asm:1} we get that $\mathcal{L}$ is of class $C^2$ and
\[
\frac{\partial\mathcal{L}}{\partial y}(y_u,u)z = Lz + b'(\cdot,y_u)z + \chi_\omega uz
\]
defines an isomorphism between $Y_{\bar p}$ and $W^{1,\bar p'}(\Omega)^*$ for all $u \in \AA$. Indeed, it is obvious that $\frac{\partial\mathcal{L}}{\partial y}(y_u,u) : Y_{\bar p} \longrightarrow W^{1,\bar p'}(\Omega)^*$ is a continuous linear mapping. The bijectivity is a consequence of the Lax-Milgram theorem and, once again, \cite[Theorem 4.2]{Stampacchia65}. Hence, a straightforward application of the implicit function theorem implies that $G$ is of class $C^2$ and \eqref{E3.4} and \eqref{E3.5} hold.
\end{proof}

Associated with the state equation \eqref{E3.2} we introduce the following bilinear distributed control problem
\begin{equation*}
	\label{bdp}
	\tag{\textup{BDP}}
	\begin{aligned}
		&\text{Minimize }\ J(u) = \frac{1}{2}\norm{ y_u - y_d }_{L^2(\Omega)}^2\\
		&\text{subject to }\ u \in \uad,
	\end{aligned}
\end{equation*}
where
\[
\uad = \{u \in L^\infty(\omega) : \alpha \le u(x) \le \beta \text{ for a.a. } x \in \omega\}
\]
with $0 \le \alpha < \beta < \infty$. For $y_d$ we assume
\begin{enumerate}[label=\textup{\bfseries(A\arabic*)},wide=0pt,resume]
	\item\label{asm:5}
$y_d \in L^2(\Omega)$ holds.
\end{enumerate}
This problem is included in the abstract framework considered in \cref{S2} by taking $X = \omega$ and $\eta$ equal to the Lebesgue measure.

The next theorem is an immediate consequence of \cref{T3.1} and the chain rule.

\begin{theorem}
	\label{T3.2}
	The reduced objective $J : \AA \to \R$
	is twice Fréchet differentiable
	and the first and second derivatives are given by
	\begin{align}
		\label{eq:first_deriv_obj}
		J'(u) \, v
		&=
		\int_\Omega (y_u - y_d) \, z_v \, \dx
		=
		-\int_{\omega} \varphi_u \, y_u \, v \, \dx,
		\\
		J''(u) (v_1, v_2)
		&=
		\label{eq:second_deriv_obj_no_adj}
		\int_\Omega \big[z_{v_1} \, z_{v_2} + (y_u - y_d) \, w_{v_1,v_2}\big] \, \dx
		\\&
		\label{eq:second_deriv_obj}
		=
		\int_\Omega\big[
		(1 - \varphi_u \, b''(\cdot, y_u)) \, z_{v_1} \, z_{v_2}\big]\,\dx
		- \int_\omega
		\varphi_u \,
		\big(
			v_1 \, z_{v_2}
			+ v_2 \, z_{v_1}
		\big)
		\,\dx
	\end{align}
	where $\varphi_u \in Y$ is the unique solution of
	\begin{equation}
		\label{eq:adjoint_equation}
		L^* \, \varphi_u + b'(\cdot, y_u) \, \varphi_u + \chi_\omega u \, \varphi_u
		=
		y_u - y_d\quad \text{in } V^*,
	\end{equation}
	and $y_u, z_{v_1}, z_{v_2}, w_{v_1,v_2}$ are defined as in \cref{T3.1}.
\end{theorem}

Using \cref{T3.1,T3.2} we infer the next result by standard arguments.

\begin{theorem}\label{T3.3}
\bdp has at least one global solution. Moreover, any local solution $\bar u$ in the sense of $L^p(\omega)$, for some $p \in [1,\infty]$, satisfies
\begin{equation}
\int_\omega\bar\varphi\bar y(u - \bar u)\, \dx \le 0 \quad \forall u \in \uad,
\label{E3.10}
\end{equation}
where $\bar y$ and $\bar\varphi$ are the state and adjoint state, respectively, corresponding to $\bar u$.
\end{theorem}

In the rest of this section, $\bar u$ will denote a fixed element of $\uad$ satisfying \eqref{E3.10}. We are going to apply the results obtained in the abstract framework in \cref{S2}. To this end, we observe that \ref{H1} obviously holds with $X = \omega$ and \ref{H4} is fulfilled with $\bar\psi = -(\bar\varphi\bar y)|_\omega$. Assumption \ref{H5} is formulated in our setting as follows: there exists a constant $K$ such that
\begin{equation}
|\{x \in \omega : |\bar\varphi(x)\bar y(x)| \le \varepsilon\}| \le K\varepsilon \ \ \forall \varepsilon > 0,
\label{E3.11}
\end{equation}
where $|\cdot|$ denotes the Lebesgue measure in $\omega$. Then, \eqref{E2.24} holds.

For the second-order analysis we introduce the cone $C_{\bar u}^\tau$ as in \eqref{eq:critical_cone}.
The rest of this section is devoted to prove that the quadratic growth condition \eqref{E2.9} holds
under the second-order condition \eqref{eq305}.
For that, we apply \cref{T2.2}.
Therefore, we only need to verify that assumptions \ref{H2} and \ref{H3} hold.
The following lemma will be used for this verification.

\begin{lemma}
	\label{L3.1}
	Given $c \in L^\infty(\Omega)$ with $c \ge 0$, we consider the equation
	\begin{equation}
	Ly + cy = f \ \ \text{in } V^*.
	\label{E3.12}
	\end{equation}
Then, the following statements hold
\begin{align}
\|y\|_{L^6(\Omega)} &\le C_L\|f\|_{L^{6/5}(\Omega)}\quad \forall f \in  L^{6/5}(\Omega),\label{E3.13}\\
\forall p > \frac{3}{2}\ \exists C_p > 0 : \|y\|_{L^\infty(\Omega)} &\le \mathrlap{C_p\|f\|_{L^p(\Omega)}}\phantom{C_L\|f\|_{L^{6/5}(\Omega)}}\quad \forall f \in L^p(\Omega),\label{E3.14}\\
\forall p \in [1,3)\  \exists C_p > 0 : \phantom{\|y\|_{L^\infty(\Omega)}}\mathllap{\|y\|_{L^p(\Omega)}} &\le \mathrlap{C_p\|f\|_{L^1(\Omega)}}\phantom{C_L\|f\|_{L^{6/5}(\Omega)}}\quad \forall f \in V^* \cap L^1(\Omega),\label{E3.15}
\end{align}
where $y \in V$ denotes the unique solution of \eqref{E3.12}.
\end{lemma}
\begin{proof}
Inequality \eqref{E3.13} is an immediate consequence of the continuous embeddings $V \subset L^6(\Omega)$ and $L^{6/5}(\Omega) \subset V^*$ for $n \le 3$. Inequality \eqref{E3.14} is proved in \cite[Theorem 4.2]{Stampacchia65}. 	We argue by transposition to prove \eqref{E3.15}.
	For an arbitrary $g \in L^{p'}(\Omega)$ with $\frac{1}{p} + \frac{1}{p'} = 1$, we denote by $z \in V$ the solution of the adjoint equation
	\begin{equation*}
		L^* \, z + c \, z  = g\ \ \text{in } V^*.
	\end{equation*}
Since $p' > \frac{3}{2}$, we can apply again \eqref{E3.14} to the adjoint equation and obtain
	\begin{equation*}
		\norm{z}_{L^\infty(\Omega)}
		\le
		C_{p'} \, \norm{g}_{L^{p'}(\Omega)}.
	\end{equation*}
	Now, we have
	\begin{align*}
		\int_\Omega y \, g \, \dx
		&=
		\dual{y}{ L^* \, z + c \, z }_{V, V^*}
		\\
		&=
		\dual{z}{ L   \, y + c \, y}_{V, V^*}
		% \\ &
		=
		\int_\Omega z \, f \, \dx
		\\ &
		\le
		\norm{z}_{L^\infty(\Omega)}
		\,
		\norm{f}_{L^1(\Omega)}
		% \\ &
		\le
		C_{p'}
		\,
		\norm{g}_{L^{p'}(\Omega)}
		\,
		\norm{f}_{L^1(\Omega)}
		.
	\end{align*}
	This implies
	$\norm{y}_{L^p(\Omega)} \le C_{p'} \, \norm{f}_{L^1(\Omega)}$.
\end{proof}

Of course, better estimates can be obtained in the previous lemma for dimensions $n < 3$, but we do not need them here.

\begin{remark}\label{R3.1}
Let us observe that the solution $z_v$ of \eqref{E3.4} satisfies the estimates \eqref{E3.13}--\eqref{E3.15} for $f = -\chi_\omega vy_u$. It is enough to take $c(x) = b'(x,y_u(x)) + \chi_\omega(x) u(x)$. Moreover, using \eqref{E3.3}, we get that $\{y_u\}_{u \in \uad}$ is uniformly bounded in $L^\infty(\Omega)$. Hence, the mentioned estimates for $z_v$ can be written in terms of the norm of $v$ in $\omega$.

Additionally, if $u_1, u_2 \in \uad$, then the estimates \eqref{E3.13}--\eqref{E3.15} are valid for $e = y_{u_2} - y_{u_1}$ in terms of $u_2 - u_1$. Indeed, it is enough to observe that subtracting the equations for $y_{u_2}$ and $y_{u_1}$, and using the mean value theorem we get that
\[
Le + b'(\cdot,y_\theta)e + \chi_\omega u_1e = \chi_\omega(u_1 - u_2)y_{u_2} \ \ \text{ in } V^*,
\]
where $y_\theta = y_{u_1} + \theta(y_{u_2}-y_{u_1})$ for some measurable function $0 \le \theta(x) \le 1$. Now, we apply \cref{L3.1} with $c(x) = b'(x,y_\theta(x)) + \chi_\omega(x)u_1(x)$ and $f = \chi_\omega(u_1 - u_2)y_{u_2}$, and we observe that $y_{u_2}$ is bounded in $L^\infty(\Omega)$.

The same comments apply to the difference of the adjoint states $\phi = \varphi_{u_2} - \varphi_{u_1}$. Indeed, $\phi$ satisfies the equation
\[
L^*\phi + b'(\cdot,y_{u_1})\phi + \chi_\omega u_1\phi = [b'(\cdot,y_{u_2}) - b'(\cdot,y_{u_1})]\varphi _{u_2} + \chi_\omega(u_2 - u_1)\varphi _{u_2}\ \ \text{ in } V^*.
\]
Besides the fact that $\varphi _{u_2} \in L^\infty(\Omega)$ we have with assumption \ref{asm:1} that
\[
\|b'(\cdot,y_{u_2}) - b'(\cdot,y_{u_1})\|_{L^r(\Omega)} \le C\|y_{u_2} - y_{u_1}\|_{L^r(\Omega)} \quad \forall r \ge 1.
\]
Then, we apply the convenient inequality of \cref{L3.1} to estimate $\|y_{u_2} - y_{u_1}\|_{L^r(\Omega)}$ in terms of $\|u_2 - u_1\|_{L^p(\omega)}$.
\end{remark}

\textit{Verification of \ref{H2}.} We prove that \ref{H2} holds with $q = \frac{6}{5}$. Since $\bar\varphi$ and $b(\cdot,\bar y)$ are bounded functions, according to the expression for $J''$ in \eqref{eq:second_deriv_obj} we only need the estimates
\begin{align*}
&\int_\Omega|z_{v_1}z_{v_2}|\, \dx \le \|z_{v_1}\|_{L^2(\Omega)}\|z_{v_2}\|_{L^2(\Omega)}\\
&\qquad \stackrel{\eqref{E3.15}}{\le} C\|v_1\|_{L^1(\omega)} \|v_2\|_{L^1(\omega)} \le C|\omega|^{1/3}\|v_1\|_{L^{6/5}(\omega)} \|v_2\|_{L^{6/5}(\omega)},
\end{align*}
and
\[
\int_\omega|v_1z_{v_2}|\, \dx \le \|v_1\|_{L^{6/5}(\omega)}\|z_{v_2}\|_{L^6(\Omega)}
\stackrel{\eqref{E3.13}}{\le}C\|v_1\|_{L^{6/5}(\omega)} \|v_2\|_{L^{6/5}(\omega)}.
\]
Hence, \ref{H2} holds with $q=6/5$.

\textit{Verification of \ref{H3}.} Let us fix $\varepsilon > 0$. For some $\delta$ that we will specify later, we take $u \in \uad \cap B^1_\delta(\bar u)$, and set $u_\theta = \bar u + \theta(u - \bar u)$ for some $\theta \in [0,1]$. Let us denote $v = u - \bar u$, $y_\theta = G(u_\theta)$, $z_\theta = G'(u_\theta)v$, and $\varphi_\theta$ the adjoint state corresponding to $u_\theta$. Analogously, we denote $(\bar y,\bar z,\bar \varphi)$  the associated functions to $\bar u$. With this notation, from \eqref{eq:second_deriv_obj} we obtain
\begin{align*}
&[J''(u_\theta) - J''(\bar u)]v^2 = \int_\Omega\big[(1 - \varphi_\theta b''(\cdot,y_\theta))z_\theta^2 - (1 - \bar\varphi b''(\cdot,\bar y))\bar z^2\big]\, \dx\\
& \qquad- 2\int_\omega(\varphi_\theta v z_\theta - \bar\varphi v \bar z)\, \dx\\
& \quad = \int_\Omega[(1 - \bar\varphi b''(\cdot,\bar y))](z_\theta^2 - \bar z^2)\, \dx + \int_\Omega(\bar\varphi - \varphi_\theta)b''(\cdot,y_\theta)z_\theta^2\, \dx\\
& \qquad + \int_\Omega\bar\varphi[b''(\cdot,\bar y) - b''(\cdot,y_\theta)]z^2_\theta\, \dx - 2\int_\omega(\varphi_\theta - \bar\varphi)v z_\theta\, \dx - 2\int_\omega\bar\varphi v (z_\theta - \bar z)\, \dx.
\end{align*}
We have to estimate these five integrals, that we denote by $I_1$ to $I_5$. From our assumption \ref{asm:1} and \eqref{E3.3} we deduce that $y_\theta$, $\bar y$, $b''(\cdot,y_\theta)$ and $b''(\cdot,\bar y)$ are bounded by a constant independent of $\theta \in [0,1]$ and $u \in \uad$. Moreover, from \cite[Theorem 4.2]{Stampacchia65} or \eqref{E3.14} and \ref{asm:5}, we infer the uniform boundedness of the adjoint states $\varphi_\theta$ and $\bar\varphi$.

As a further preparation, we provide an estimate for the difference $e = z_\theta - \bar z$.
By taking the difference of the corresponding equations \eqref{E3.4}, we find that $e$ solves the equation
\begin{equation*}
	Le + \bpd{\bar y}e + \chi_\omega\bar u e
	=
	(\bpd{\bar y} - \bpd{y_\theta}) z_\theta
	+ \chi_\omega(\bar u - u_\theta) z_\theta
	+ (\bar y - y_\theta) v.
\end{equation*}
Owing to \cref{L3.1}, we can estimate $\norm{e}_{L^6(\Omega)}$
by the $L^{6/5}(\Omega)$ norm of the right-hand side.
Together with Hölder's inequality, we obtain the estimate
\begin{align*}
	\norm{z_\theta - \bar z}_{L^6(\Omega)}
	&\le
	C_L\norm{\bpd{\bar y} - \bpd{y_\theta}}_{L^{12/5}(\Omega)} \norm{z_\theta}_{L^{12/5}(\Omega)}
	\\&\qquad
	+ C_L\norm{\bar u - u_\theta}_{L^{3/2}(\omega)} \norm{z_\theta}_{L^{6}(\Omega)}
	+ C_L\norm{\bar y - y_\theta}_{L^{6}(\Omega)} \norm{v}_{L^{3/2}(\omega)}
	.
\end{align*}
Now, we can use \ref{asm:1} and \cref{R3.1}, and we arrive at
\begin{align*}
	\norm{z_\theta - \bar z}_{L^6(\Omega)}
	&\le
	C \norm{\bar u - u_\theta}_{L^1(\omega)} \norm{v}_{L^{1}(\omega)}
	+ C \norm{\bar u - u_\theta}_{L^{3/2}(\omega)} \norm{v}_{L^{6/5}(\omega)}
	\\&\qquad
	+ C \norm{\bar u - u_\theta}_{L^{6/5}(\omega)} \norm{v}_{L^{3/2}(\omega)}
	.
\end{align*}
Using $u_\theta - \bar u = \theta v$,
taking into account that $\norm{u_\theta - \bar u}_{L^1(\omega)} \le\norm{v}_{L^1(\omega)} \le \delta $ and that
\begin{equation}
	\label{eq:3.16}
	\norm{v}_{L^q(\omega)}
	\le \norm{v}_{L^1(\omega)}^{1/q} \norm{v}_{L^\infty(\omega)}^{1-1/q}
	\le C \norm{v}_{L^1(\omega)}^{1/q},
\end{equation}
for $\delta \le 1$ the above estimates becomes
\begin{equation}
	\label{E3.16}
	\norm{z_\theta - \bar z}_{L^6(\Omega)}
	\le
	C \norm{v}_{L^1(\omega)}^{3/2}.
\end{equation}
Now, we are in position to estimate the above integrals.
For the first integral, we have
\begin{align*}
	\abs{ I_1 }
	&=\Bigabs{ \int_\Omega[(1 - \bar\varphi b''(\cdot,\bar y))](z_\theta^2 - \bar z^2)\, \dx }
	\\&
	\le \norm{1 - \bar\varphi b''(\cdot,\bar y)}_{L^{\infty}(\Omega)}
	\norm{z_\theta + \bar z}_{L^{2}(\Omega)}
	\norm{z_\theta - \bar z}_{L^{2}(\Omega)}
	% \\&
	\le
	C \, \norm{v}_{L^1(\omega)} \norm{v}_{L^1(\omega)}^{3/2},
\end{align*}
where we used \cref{R3.1} and \eqref{E3.16}.
Next,
\begin{align*}
	\abs{I_2} &=\Bigabs{\int_\Omega(\bar\varphi - \varphi_\theta)b''(\cdot,y_\theta)z_\theta^2\, \dx}
	% \\&
	\le C \, \norm{\bar\varphi - \varphi_\theta}_{L^{6}(\Omega)} \norm{b''(\cdot,y_\theta)}_{L^{\infty}(\Omega)} \norm{z_\theta}_{L^{6}(\Omega)}^2
	\\&
	\le C \, \norm{v}_{L^{6/5}(\omega)}^3
	\le C \, \norm{v}_{L^1(\omega)}^{5/2},
\end{align*}
where again \cref{R3.1} and \eqref{eq:3.16} have been utilized.
For the next integral, we remark that
$\norm{b''(\cdot,\bar y) - b''(\cdot,y_\theta)}_{L^\infty(\Omega)}$
can be estimated by any small positive number
if $\norm{\bar y - y_\theta}_{L^\infty(\Omega)}$ is small enough, cf.\ \ref{asm:1}.
For this, it is sufficient that $\delta$ is small enough,
since $u_\theta \in B_\delta^1(\bar u) \cap \uad$, see again \cref{R3.1}.
This along with \eqref{E3.15} leads to the estimate
\begin{align*}
	\abs{I_3} &=\Bigabs{\int_\Omega\bar\varphi[b''(\cdot,\bar y) - b''(\cdot,y_\theta)]z^2_\theta\, \dx}
	% \\&
	\le \norm{\bar\varphi}_{L^{\infty}(\Omega)} \norm{b''(\cdot,\bar y) - b''(\cdot,y_\theta)}_{L^{\infty}(\Omega)} \norm{z_\theta}_{L^{2}(\Omega)}^2
	\\&
	\le \frac{\varepsilon}{5} \, \norm{v}_{L^1(\omega)}^2.
\end{align*}
Finally, we obtain by using similar arguments
the estimates
\begin{align*}
	\abs{I_4} &=\Bigabs{\int_\omega(\varphi_\theta - \bar\varphi)v z_\theta\, \dx}
	% \\&
	\le \norm{\varphi_\theta - \bar\varphi}_{L^{6}(\Omega)} \norm{v}_{L^{3/2}(\omega)} \norm{z_\theta}_{L^{6}(\Omega)}
	\\&
	\le C \norm{v}_{L^{6/5}(\omega)} \norm{v}_{L^{3/2}(\omega)} \norm{v}_{L^{6/5}(\omega)}
	\le C \norm{v}_{L^1(\omega)}^{5/2}
\end{align*}
and
\begin{align*}
	\abs{I_5} &=\Bigabs{\int_\omega\bar\varphi v (z_\theta - \bar z)\, \dx}
	% \\&
	\le \norm{\bar\varphi}_{L^{\infty}(\Omega)} \norm{v}_{L^{6/5}(\omega)} \norm{z_\theta - \bar z}_{L^{6}(\Omega)}
	\\&
	\le C \norm{v}_{L^{6/5}(\omega)} \norm{v}_{L^{1}(\omega)}^{3/2}
	\le C \norm{v}_{L^{1}(\omega)}^{7/3},
\end{align*}
where we used additionally \eqref{E3.16}.
Putting these inequalities together, we obtain the desired estimate
\begin{equation*}
	\bigabs{ [J''(u_\theta) - J''(\bar u)]v^2 }
	\le
	\abs{I_1} + \abs{I_2} + \abs{I_3} + \abs{I_4} + \abs{I_5}
	\le
	\varepsilon \, \norm{u - \bar u}_{L^1(\omega)}^2,
\end{equation*}
if $\delta > 0$ is chosen small enough.
Hence, we verified
\ref{H3}
in our current setting.

\textit{Application of \cref{T2.2}.}
We have verified that the assumptions \ref{H1}--\ref{H4}
are satisfied in the setting of the bilinear distributed control problem \bdp.
Thus, we can apply \cref{T2.2} and we obtain the following sufficient second-order condition.
\begin{theorem}
	\label{T3.4}
	Let us assume that \ref{asm:coercivity}--\ref{asm:5} are satisfied.
	Moreover, we suppose that there is a constant $K > 0$, such that \eqref{E3.11} holds
	and that there exist $\tau > 0$ and $\kappa' < 2 \, \kappa$
	such that
	\begin{equation}
		J''(\bar u)v^2 \ge -\kappa'\|v\|^2_{L^1(\omega)} \ \ \forall v \in C^\tau_{\bar u},
		\label{E3.20}
	\end{equation}
	where $\kappa = (4(\beta - \alpha)K)^{-1}$.
	Then, there exist $\nu > 0$ and $\delta > 0$ such that
	\begin{equation*}
		J(\bar u) + \nu\|u - \bar u\|^2_{L^1(\omega)} \le J(u) \ \ \ \forall u \in \uad \cap B^1_\delta(\bar u).
\end{equation*}
\end{theorem}

\subsection{A bilinear boundary control problem}
\label{S32}

In this section we assume that $n = 2$.
We outline the main steps which are necessary to transfer the analysis
of \cref{S31} to a bilinear boundary control problem.
We follow the notation introduced in \cref{S3} and assume that \ref{asm:coercivity}--\ref{asm:3} hold.
Further, we take $\omega = \Gamma_N$ equipped with the surface measure.
We define the operator $S_\omega:L^2(\omega) \longrightarrow V^*$ by
\[
\langle S_\omega(g),z\rangle = \int_\omega g(x)z(x)\, \dx\quad \forall z \in V,
\]
where we are denoting the trace of $z$ on $\omega$ by $z$ as well. It is well known that there exist a constant $C_\omega$ depending on $\Omega$ such that
\begin{equation}
\|z\|_{L^2(\omega)} \le C_\omega\|z\|_V\quad \forall z \in V.
\label{E3.21}
\end{equation}
Now, we consider the state equation
\begin{equation}
Ly + b(\cdot,y) + S_\omega(u y) = f \ \ \text{ in } V^*,
\label{E3.22}
\end{equation}
with $u \in \AA$. Here, $\AA$ is defined as follows
	\[
\AA = \{v \in L^\infty(\omega) : \exists\varepsilon_v > 0 \text{ such that } v(x) > -\frac{\Lambda	}{2C_\omega^2} + \varepsilon_v\ \text{ for a.a. } x \in \omega\},
\]
where $\Lambda$ was introduced in \ref{asm:coercivity}. From the assumptions \ref{asm:coercivity} and \ref{asm:3} along with \eqref{E3.21} we get
\[
\langle Ly,y\rangle + \langle S_\omega(uy),y\rangle \ge \Lambda\|y\|^2_V - \frac{\Lambda	}{2C_\omega^2}\|y\|^2_{L^2(\omega)} \ge \frac{\Lambda}{2}\|y\|^2_V \quad \forall y \in V.
\]

Then, \cref{T3.1} holds with the obvious modifications.
In particular, the equations \eqref{E3.4} and \eqref{E3.5} are modified as follows
\begin{equation}
Ly + b'(\cdot,y_u)z_v + S_\omega(uz_v) + S_\omega(vy_u) = 0
\label{E3.23}
\end{equation}
and
\begin{equation}
\begin{aligned}
&L \, w_{v_1,v_2}+ b'(\cdot, y_u) \, w_{v_1,v_2}+ S_\omega(u \, w_{v_1,v_2})\\
&\qquad\qquad + b''(\cdot, y_u) \, z_{v_1} \, z_{v_2}+ S_\omega (v_1 \, z_{v_2}) + S_\omega(v_2 \, z_{v_1}) = 0.
\end{aligned}
\label{E3.24}
\end{equation}

Associated with the state equation \eqref{E3.2} we introduce the bilinear boundary control problem
\begin{equation}
	\label{bbp}
	\tag{\textup{BBP}}
	\begin{aligned}
		&\text{Minimize }\ J(u) = \frac{1}{2}\norm{ y_u - y_d }_{L^2(\Omega)}^2\\
		&\text{subject to }\ u \in \uad,
	\end{aligned}
\end{equation}
where
\[
\uad = \{u \in L^\infty(\omega) : \alpha \le u(x) \le \beta \text{ for a.a. } x \in \omega\}
\]
with $0 \le \alpha < \beta < \infty$. We suppose that $y_d$ satisfies the assumption \ref{asm:5}. Then, \cref{T3.2} holds, we only need to change the adjoint state equation \eqref{eq:adjoint_equation} by
\begin{equation}
\label{E3.25}
L^* \, \varphi_u + b'(\cdot, y_u) \,\varphi_u + S_\omega(u \, \varphi_u) =  y_u - y_d\quad \text{in } V^*.
\end{equation}
We also have that \cref{T3.3} holds. To get the sufficient second-order conditions we assume that \eqref{E3.11} is fulfilled. Then, to check that \cref{T2.1,T2.2} hold we need to check that assumptions \ref{H1}--\ref{H5} are satisfied. As in \cref{S31}, it is enough to verify \ref{H2} and \ref{H3}. To this end we will use the following lemma.

\begin{lemma}
Let $c \in L^\infty(\Omega)$ be nonnegative and $u \in \AA$. For $(f,g) \in L^2(\Omega) \times L^2(\omega)$ let $y \in V$ be the solution of the equation
\begin{equation}
Ly + cy + S_\omega(uy) = f + S_\omega(g) \ \text{ in } V^*.
\label{E3.26}
\end{equation}
Then, for every $p \in [1,\infty)$ and $q > 1$ there exist constants $C_p$ and $M_q$ independent of $(f,g)$, $c$ and $u$ such that
\begin{align}
&\|y\|_{L^p(\Omega)} \le C_p\big(\|f\|_{L^1(\Omega)} + \|g\|_{L^1(\omega)}\big),
\label{E3.27}\\
&\|y\|_{L^\infty(\Omega)} \le M_q\big(\|f\|_{L^q(\Omega)} + \|g\|_{L^q(\omega)}\big).
\label{E3.28}
\end{align}
\label{L3.2}
\end{lemma}

\begin{proof}
Since $L^1(\Omega)$ and $L^1(\omega)$ are subspaces of the space of real and regular Borel measures in $\Omega$ and $\omega$, respectively, we can apply the well known results for measures to deduce that the solution $y$ of \eqref{E3.26} satisfies
\[
\|y\|_{W^{1,s}(\Omega)} \le C_s\big(\|f\|_{L^1(\Omega)} + \|g\|_{L^1(\omega)}\big)
\]
for every $s \in [1,\frac{n}{n-1})$ and some constant $C_s$ independent of $(f,g)$, $c$ and $u$; see, for instance, \cite{AR97}, \cite{Casas93}, or \cite{MeyerPanizziSchiela2011}.

Since we have assumed $n = 2$, for every $p \in [1,\infty)$ there exists $s < \frac{n}{n-1}$ such that $W^{1,s}(\Omega) \subset L^p(\Omega)$ and, hence, \eqref{E3.27} follows from the above estimate. The estimate \eqref{E3.28} is proved in \cite[Theorem~2]{AR97}.
\end{proof}

Hence, though simpler estimates can be used, the estimates  used in \cref{S31} are valid to verify  \ref{H2} and \ref{H3}.
As a consequence, we obtain a second-order sufficient condition
analogously to \cref{T3.4} in the distributed case.

We finally mention that the same technique cannot be used to address the case $n > 2$.
The verification of \ref{H2} and \ref{H3} for bilinear boundary control problems in more than two spatial dimensions
remains an open problem.

\section{Numerical approximation of distributed control problems}
\label{S4}
In this section, we consider the following boundary value problem
\begin{equation}
\left\{\begin{array}{l}
\displaystyle Ay + b(\cdot,y) + \chi_\omega uy = f \ \text{ in } \Omega,\\y = 0 \ \text{ on } \Gamma,\end{array}\right.\label{E4.1}
\end{equation}
where $A$ is given by \eqref{E3.0} with coefficients $a_{ij} \in C^{0,1}(\bar\Omega)$ satisfying the ellipticity condition
\[
\sum_{i,j = 1}^na_{ij}(x)\xi_i\xi_j \ge \Lambda|\xi|^2 \ \ \forall x \in \Omega \text{ and } \forall \xi \in \R^n.
\]
We also assume that $a_0 \in L^\infty(\Omega)$, $a_0 \ge 0$, $b$ satisfies the assumption \ref{asm:1}, and $f \in L^{\bar p}(\Omega)$ with $\bar p > n$. We follow the notation introduced in \cref{S3}. Hence, by \cref{T3.1} we know that \eqref{E4.1} has a unique solution $y_u \in Y = H^1_0(\Omega) \cap L^\infty(\Omega)$ $\forall u \in \AA$.

We also introduce the adjoint state equation associated to the control $u$
\begin{equation}
\left\{\begin{array}{l}
\displaystyle A^*\varphi + b'(\cdot,y_u)\varphi + \chi_\omega u\varphi = y_u - y_d \ \text{ in } \Omega,\\\varphi = 0 \ \text{ on } \Gamma.\end{array}\right.\label{E4.2}
\end{equation}

Now, we consider the control problem \bdp associated to the equation \eqref{E4.1}.
Here we suppose that $y_d \in L^{\bar p}(\Omega)\cap L^2(\Omega)$.
We also assume that $\bar\omega \subset \Omega$.
Let us observe that if this condition does not hold, then the assumption \eqref{E3.11} can be fulfilled only in some extreme cases. This is due to the fact that $\bar y$ and $\bar\varphi$ vanish on $\Gamma$ and, hence, the $\{x \in \Omega : |\bar y(x)\bar\varphi(x)| \le \varepsilon\}$ contains a strip along the boundary with a measure of order $\sqrt{\varepsilon}$. The situation is different for Neumann boundary problems.

Since assumptions \ref{asm:coercivity}--\ref{asm:5} are satisfied, \cref{T3.2,T3.3} are valid for the the control problem \bdp associated to the state equation \eqref{E4.1}. In what follows, $\bar u$ will denote a local solution of \bdp satisfying the regularity condition \eqref{E3.11}. Therefore, \cref{T3.4} holds as well.

The goal of this section is to prove error estimates for the numerical approximation of \bdp based on a finite element discretization.
To this end, we assume that $\Omega$ is convex and $\Gamma$ is of class $C^{1,1}$.
Therefore, we have additional regularity for the states $y_u$ and adjoint states $\varphi_u$ for every $u  \in \AA$, namely $y_u, \varphi_u \in W^{2,\bar p}(\Omega) \cap W_0^{1,\bar p}(\Omega)$; see \cite[Chapter 2]{Grisvard85}.
Since $\bar p > n$, we have that $W^{2,\bar p}(\Omega) \subset C^1(\bar\Omega)$. If $n = 2$, this regularity holds for a convex and polygonal domain $\Omega$ assuming that the coefficients $a_{ij}$ are of class $C^1$ in $\bar\Omega$. In dimension $n = 3$, the regularity result is valid for rectangular parallelepipeds under the same $C^1$ regularity of the coefficients; see \cite[Chapter 4]{Grisvard85}, \cite[Corollary 3.14]{Dauge92}.

Let $\{{\mathcal{T}}_h\}_{h>0}$ be a quasi-uniform family of triangulations of $\bar\Omega$; see \cite{Ciarlet78}. We set $\overline\Omega_h=\cup_{T\in {\mathcal{T}}_h}T$ with $\Omega_h$ and $\Gamma_h$ being its interior and boundary, respectively. We assume that the vertices of  ${\mathcal{T}}_h$ placed on the boundary $\Gamma_h$ are also points of $\Gamma$ and there exists a constant $C_\Gamma > 0$ such that $\operatorname{dist}(x,\Gamma) \le C_\Gamma h^2$ for every $x \in \Gamma_h$. This always holds if $\Gamma$ is a $C^2$ boundary and $n = 2$. From this assumption we know \cite[Section~5.2]{Raviart-Thomas83} that
\begin{equation}
|\Omega\setminus \Omega_h| \leq C_\Omega h^2, \label{E4.3}
\end{equation}
where $|\cdot|$ denotes the Lebesgue measure. Let us denote by $\mathcal{T}_{\omega,h}$ the family of all elements $T \in \mathcal{T}_h$ such that $T \subset \bar\omega$. We set $\bar\omega_h = \bigcup_{T\in \mathcal{T}_{\omega,h}}T$ and $\omega_h$ is its interior. We also assume that $|\omega\setminus\omega_h| \le C_\omega h^{p_\omega}$ with $p_\omega > n/2$.

Associated with this triangulation we define the spaces
\begin{align*}
& \uh = \{u_h \in L^\infty(\omega_h) : {u_h}_{\mid_T} \in \mathcal{P}_0(T)\ \ \forall T \in \mathcal{T}_{\omega,h}\},\\
&\yh = \{y_h \in C(\bar\Omega) : {y_h}_{\mid_T} \in \mathcal{P}_1(T)\ \ \forall T \in \mathcal{T}_h\text{ and } y_h = 0 \text{ in } \bar\Omega\setminus\Omega_h\},
\end{align*}
where $\mathcal{P}_k(T)$ denotes the polynomial of degree $k$ in $T$ with $k = 0, 1$. Now, for every $u \in \AA$ we consider the discrete system of nonlinear equations
\begin{align}
&\text{Find } y_h \in Y_h \text{ such that } \forall z_h \in Y_h\notag\\
&a(y_h,z_h) + \int_\Omega[b(\cdot,y_h) + \chi_{\omega_h} uy_h]z_h\, \dx = \int_\Omega fz_h\, \dx,
\label{E4.4}
\end{align}
where the bilinear form $a$ was defined in \eqref{E3.a}. Using our assumptions on $b$ and the ellipticity of the operator $y \to Ay + uy$, the existence and uniqueness of a solution of \eqref{E4.3} follows by standard arguments. This solution will be denoted by $y_h(u)$. We also consider the discrete adjoint state equation
\begin{align}
&\text{Find } \varphi_h \in Y_h \text{ such that } \forall z_h \in Y_h\notag\\
&a(z_h,\varphi_h) + \int_\Omega[a_0 + b'(\cdot,y_h(u)) + \chi_{\omega_h} u]\varphi_hz_h\, \dx = \int_\Omega (y_h(u) - y_d)z_h\, \dx.
\label{E4.5}
\end{align}
The solution of this adjoint equation is denoted by $\varphi_h(u)$.

%The following estimates can be deduced directly from
%\cite[Lemmas 4.4 and 4.5]{Casas02} and \cite{Rannacher76,Schatz98}.

The following approximation results are needed for the numerical analysis of the discrete control problem.

\begin{lemma}\label{L4.1} Let $u \in \AA$ fulfill $\|u\|_{L^\infty(\omega)}\leq M$, and let $y$, $y_h$, $\varphi$ and $\varphi_{h}$ be the solutions of \eqref{E4.1}, \eqref{E4.4}, \eqref{E4.2} and \eqref{E4.5}, respectively. Then, for some constant $C$ depending on $M$ we have
\begin{equation}\label{E4.6}
\|y-y_h\|_{L^\infty(\Omega)} + \|\varphi-\varphi_h\|_{L^\infty(\Omega)} \leq C h,
\end{equation}
\end{lemma}

\begin{proof}
Let us denote $u_h = \chi|_{\omega_h}u$ and $y_{u_h}$ its continuous associated state. From \cref{L3.1} and \cref{R3.1}, and using the classical $L^\infty$-estimates for finite element approximations, see \cite{ACT02,Casas-Mateos02a} and \cite{Rannacher76,Schatz98}, we get
\begin{align*}
&\|y - y_h\|_{L^\infty(\Omega)} \le \|y - y_{u_h}\|_{L^\infty(\Omega)} + \|y_{u_h} - y_h\|_{L^\infty(\Omega)}\\
&\le C_1(\|u - u_h\|_{L^{^{p_\omega}}(\omega)} + h^{2 - n/\bar p}|\log{h}|) \le Ch,
\end{align*}
where we have used that $|\omega\setminus\omega_h| \le C_\omega h^{p_\omega}$. From this estimate we deduce the corresponding estimate for $\varphi - \varphi_h$ by using similar arguments.
\end{proof}

Finally, we define the discrete control problem
\begin{equation*}
	\label{bdph}
	\tag{\textup{BDP}$_h$}
	\begin{aligned}
		&\text{Minimize }\ J_h(u_h) = \frac{1}{2}\norm{ y_h(u_h) - y_d }_{L^2(\Omega_h)}^2
		+ \frac{\alpha_h}{2} \, \norm{u_h}_{L^2(\omega_h)}^2\\
		&\text{subject to }\ u_h \in \uadh,
	\end{aligned}
\end{equation*}
where
\[
\uadh = \{u_h \in \uh : \alpha \le u_h(x) \le \beta \text{ for a.a. } x \in \omega_h\}.
\]
Moreover, we included a Tikhonov parameter $\alpha_h \ge 0$
and require $\alpha_h \to 0$ as $h \to 0$.
This regularization term is beneficial for the numerical solution of \eqref{bdph}
and we will prove that the choice $\alpha_h = c \, h$ yields the same order of convergence
as $\alpha_h = 0$, see \eqref{E4.11} below.

Let us check that these approximations of \bdp fit into the framework described in \cref{S23}. To this end we have to check the assumptions \ref{D1}--\ref{D4}.  First, we observe that taking $X = \omega$, $X_h = \omega_h$ and $\eta =$ Lebesgue measure in $\omega$, \ref{D1} follows from our assumption $|\omega \setminus \omega_h| \to 0$ as $h \to 0$.

Assumption \ref{D1p5} is immediate. Indeed, it is enough to observe that given $u \in \uad$ we can take $u_h$ as the projection of $u$ on $\uh$:
\begin{equation}
\Pi_h u = u_h = \sum_{T \in \mathcal{T}_{\omega,h}}u_T\chi_T\  \text{ with }\ u_T = \frac{1}{T}\int_Tu\, \dx,
\label{E4.9}
\end{equation}
where $\chi_T$ denotes the characteristic function of $T$. It is well known that $u_h \to u$ strongly in $L^p(\omega)$ under the assumption $u \in L^p(\omega)$; see \cite{DDW75}.

Now, \ref{D2} is obvious.
\ref{D3} is a straightforward consequence of the following lemma.
\begin{lemma}\label{L4.2}
If $u_h\rightharpoonup u$ weakly in $L^1(\omega)$ with $u_h \in \AA \cap \uh$ and $u \in \AA$, and there exists a constant $M > 0$ such that $\|u_h\|_{L^\infty(\omega_h)} \le M$ $\forall h > 0$, then $y_h(u_h)\to y_u$ and $\varphi_{h}(u_h)\to \varphi_{u}$ in $L^\infty(\Omega)$ as $h \to 0$ strongly, and $J(u) = \lim_{h\to 0}J_h(u_h)$.
\end{lemma}

\begin{proof}
Let us extend every $u_h$ to $\omega$ by setting $u_h(x) = 0$ $\forall x \in \omega\setminus\omega_h$. From \eqref{E4.6} we get
\[
\|y_u - y_h(u_h)\|_{L^\infty(\Omega)} \le \|y_u - y_{u_h}\|_{L^\infty(\Omega)}  + \|y_{u_h} - y_h(u_h)\|_{L^\infty(\Omega)} \le \|y_u - y_{u_h}\|_{L^\infty(\Omega)} + Ch.
\]
Now, we prove that $\|y_u - y_{u_h}\|_{L^\infty(\Omega)} \to 0$ as $h \to 0$. Since $\|u_h\|_{L^\infty(\omega)} \le M$ $\forall h > 0$, then $\{y_{u_h}\}_h$ is bounded in $W^{2,\bar p}(\Omega)$. Using the compactness of the embedding $W^{2,\bar p}(\Omega) \subset L^\infty(\Omega)$, we deduce easily the convergence $\|y_u - y_{u_h}\|_{L^\infty(\Omega)} \to 0$ as $h \to 0$. The convergence $J_h(u_h) \to J(u)$ follows easily by using $\alpha_h \to 0$.
\end{proof}

To check \ref{D4} we take
\[
\AA_h = \{v \in L^\infty(\omega_h) : \exists\varepsilon_v > 0 \text{ such that } v(x) > -\frac{\Lambda	}{2} + \varepsilon_v\ \text{ for a.a. } x \in \omega_h\}.
\]
It is easy to prove that $J_h:\AA_h \longrightarrow \R$ is of class $C^2$ and its first derivative is given by
\begin{equation}
J'_h(u)v = -\int_{\omega_h}\varphi_h(u)y_h(u)\, \dx
+ \alpha_h \, \int_{\omega_h} u \, v \, \dx
\quad \forall u \in \AA_h \text{ and } \forall v \in L^\infty(\omega_h),
\label{E4.10}
\end{equation}
where $y_h(u)$ and $\varphi_h(u)$ are the solutions of  \eqref{E4.4} and \eqref{E4.5}, respectively.
Hence, it is enough to take $\psi_h = - (\varphi_h(u)y_h(u))|_{\omega_h} + \alpha_h \, u$.
Concerning the function $J : \AA \to \R$, we already know that it is of class $C^2$ (\cref{T3.2}), and according to \eqref{eq:first_deriv_obj} we can take $\psi = -(\varphi_uy_u)|_{\omega}$.

Therefore, \cref{T2.3,T2.4} hold. Observe that \cref{T2.3} is formulated as follows.

\begin{theorem}\label{T4.1}
Assume that \ref{asm:coercivity}--\ref{asm:5} hold. For every $h$, the problem \bdph has at least a global solution $\bar u_h$.
If $\{\bar u_h\}_h$ is a sequence of global solutions of \bdph and
$\bar u_h \weaklystar \tilde u$ in $L^\infty(\omega)$,
then $\tilde u$ is a global solution of \bdp.
Conversely, if $\bar u$ is a bang-bang strict local minimum of \bdp in the $L^1(\omega)$ sense, then there exists a sequence $\{\bar u_h\}_h$ of local minimizers of problems \bdph with respect to the same topology such that $\bar u_h \to \bar u$ in $L^1(\omega)$.
\end{theorem}

Now, we apply \cref{T2.4} to get the following result.
\begin{theorem}\label{T4.2}
Assume that \ref{asm:coercivity}--\ref{asm:5} hold. Additionally, we suppose that \eqref{E3.11} is fulfilled and $\bar u$ satisfies the second-order condition \eqref{E3.20} with $\kappa' \in (0, \kappa)$. Let $\{\bar u_h\}_h$ be a sequence of local solutions of problems \bdph converging to $\bar u$ in $L^1(\omega)$. Then, there exists a constant $C$ independent of $h$ such that
\begin{equation}
\|\bar u - \bar u_h\|_{L^1(\omega_h)} \le C \, (h + \alpha_h).
\label{E4.11}
\end{equation}
\end{theorem}

\begin{proof}
To prove this theorem we will estimate the three terms in the right hand side of \eqref{E2.30}. First, we observe that
\begin{equation}
	\label{eq:estimate_Jh_J}
	\begin{aligned}
		\|J'_h(\bar u_h) - J'(\hatbar u_h)\|_{L^\infty(X_h)} &= \|\bar\varphi_h\bar y_h - \varphi_{\hatbar u_h}y_{\hatbar u_h} + \alpha_h \, u_h\|_{L^\infty(\omega_h)}
		\\
		&\le
		\|\bar\varphi_h\bar y_h - \varphi_{\hatbar u_h}y_{\hatbar u_h}\|_{L^\infty(\omega_h)} + C_0 \, \alpha_h
		,
	\end{aligned}
\end{equation}
where  $\bar y_h$ and $\bar\varphi_h$ are the discrete state and adjoint state associated with $\bar u_h$, and $y_{\hatbar u_h}$ and $\varphi_{\hatbar u_h}$ are the continuous state and adjoint state corresponding to $\hatbar u_h$, which is the extension of $\bar u_h$ to $\omega$ by $\bar u$. Now using \cref{L4.1} we obtain
\begin{align}
\|\bar\varphi_h\bar y_h - \varphi_{\hatbar u_h}y_{\hatbar u_h}\|_{L^\infty(\omega_h)} &\le \|\bar\varphi_h\|_{L^\infty(\omega_h)}\|\bar y_h - y_{\hatbar u_h}\|_{L^\infty(\omega_h)}\notag\\
&+ \|y_{\hatbar u_h}\|_{L^\infty(\omega_h)}\|\bar\varphi_h -  \varphi_{\hatbar u_h}\|_{L^\infty(\omega_h)} \le C_1h.
\label{E4.12}
\end{align}

Now, we estimate the second term of  \eqref{E2.30}. To this end, we take $u_h$ as the projection of $\bar u$ on $\uh$; see \eqref{E4.9}.
Since $\bar u$ is bang-bang by assumption, it holds $\bar u=u_h$ on all elements,
where $\bar u$ is constant.
It remains to estimate $|u_h-\bar u|$ on elements $T$, where $\bar u$ takes the values $\alpha$ and $\beta$ on some points of $T$.
Let us denote the family of such elements by $\mathcal{T}_{h,\bar u}$.
Let us take $T\in \mathcal{T}_{h,\bar u}$.
This means that $\bar\varphi\bar y$ changes the sign in $T$.
Since $\bar\varphi\bar y$ is continuous in $\bar\Omega$, there exists a point $\xi_T \in T$ such that $\bar\varphi(\xi_T)\bar y(\xi_T) = 0$. Since $\bar\varphi\bar y \in W^{2,\bar p}(\Omega) \subset C^1(\bar\Omega)$, we get the existence of constant $\bar L$ such that
\[
|\bar\varphi(x)\bar y(x)| = |\bar\varphi(x)\bar y(x) - \bar\varphi(\xi_T)\bar y(\xi_T)| \le \bar L |x - \xi_T| \le \bar L h\ \ \forall x\in T.
\]
This inequality implies that
\[
\bigcup_{T \in \mathcal{T}_{h,\bar u}}T \subset \{x \in \omega_h : |\bar\varphi(x)\bar y(x)| \le \bar L h\}.
\]
This along with \eqref{E3.11} leads to
\[
\sum_{T \in \mathcal{T}_{h,\bar u}}|T| \le K\bar L h.
\]
Hence, we infer
\begin{equation}
\|u_h - \bar u\|_{L^1(\omega_h)} = \sum_{T \in \mathcal{T}_{h,\bar u}}\|u_h - \bar u\|_{L^1(T)} \le (\beta - \alpha)K\bar L h = C_2h.
\label{E4.13}
\end{equation}
We finish the proof with the estimate of the third term of \eqref{E2.30}.
Note that by construction it holds $\hat u_h=\bar u$ on $\omega\setminus \omega_h$.
Using that $u_h$ is the projection of $\bar u$ we get with \eqref{eq:first_deriv_obj} and \eqref{E4.13}
\begin{align}
&|J'(\hatbar u_h)(\hat u_h - \bar u)| = \Big|\int_{\omega}\varphi_{\hatbar u_h}y_{\hatbar u_h}(\hat u_h - \bar u)\, \dx\Big| = \Big|\int_{\omega_h}\varphi_{\hatbar u_h}y_{\hatbar u_h}(u_h - \bar u)\, \dx\Big|\notag\\
&= \Big|\int_{\omega_h}(\varphi_{\hatbar u_h}y_{\hatbar u_h} - \Pi_h(\varphi_{\hatbar u_h}y_{\hatbar u_h}))(u_h - \bar u)\, \dx\Big|\notag\\
& \le \|\varphi_{\hatbar u_h}y_{\hatbar u_h} - \Pi_h(\varphi_{\hatbar u_h}y_{\hatbar u_h})\|_{L^\infty(\omega_h)}\|u_h - \bar u\|_{L^1(\omega_h)}\notag\\
& \le Ch\|\varphi_{\hatbar u_h}y_{\hatbar u_h}\|_{C^1(\bar\Omega)}C_2h \le C_3h^2.
\label{E4.14}
\end{align}
Here, we used that $\{\varphi_{\hatbar u_h}\}$ and $\{y_{\hatbar u_h}\}$ are uniformly bounded in $W^{2,p}(\Omega)$.
Finally, \eqref{E4.11} follows from \eqref{E2.30}, \eqref{eq:estimate_Jh_J}--\eqref{E4.14} and Young's inequality.
\end{proof}

\bibliographystyle{siam}
\bibliography{CWW}

\begin{thebibliography}{10}

\bibitem{AR97}
{\sc J.~Alibert and J.~Raymond}, {\em Boundary control of semilinear elliptic
  equations with discontinuous leading coefficients and unbounded controls},
  Numer. Funct. Anal. and Optimiz., 18 (1997), pp.~235--250.

\bibitem{ACT02}
{\sc N.~Arada, E.~Casas, and F.~{Tr\"{o}ltzsch}}, {\em Error estimates for the
  numerical approximation of a semilinear elliptic control problem}, Comput.
  Optim. Appls., 23 (2002), pp.~201--229.

\bibitem{AronnaBonnansKroner16}
{\sc M.~S. Aronna, F.~Bonnans, and A.~Kröner}, {\em Optimal control of
  infinite dimensional bilinear systems: Application to the heat and wave
  equations}.
\newblock arXiv:1602.06469, 2016.

\bibitem{AronnaBonnansGoh2016}
{\sc M.~S. Aronna, J.~F. Bonnans, and B.~S. Goh}, {\em Second order analysis of
  control-affine problems with scalar state constraint}, Math. Program., 160
  (2016), pp.~115--147.

\bibitem{Casas93}
{\sc E.~Casas}, {\em Boundary control of semilinear elliptic equations with
  pointwise state constraints}, SIAM J. Control Optim., 31 (1993),
  pp.~993--1006.

\bibitem{Casas02}
\leavevmode\vrule height 2pt depth -1.6pt width 23pt, {\em Error estimates for
  the numerical approximation of semilinear elliptic control problems with
  finitely many state constraints}, ESAIM:COCV, 8 (2002), pp.~345--374.

\bibitem{Casas2012}
\leavevmode\vrule height 2pt depth -1.6pt width 23pt, {\em Second order
  analysis for bang-bang control problems of {PDE}s}, SIAM J. Control Optim.,
  50 (2012), pp.~2355--2372.

\bibitem{Casas-Mateos02a}
{\sc E.~Casas and M.~Mateos}, {\em Uniform convergence of the {FEM}.
  {A}pplications to state constrained control problems}, Comput. Appl. Math.,
  21 (2002), pp.~67--100.

\bibitem{CWW2017}
{\sc E.~Casas, D.~Wachsmuth, and G.~Wachsmuth}, {\em Sufficient second-order
  conditions for bang-bang control problems}, preprint, TU Chemnitz, 2016.

\bibitem{Ciarlet78}
{\sc P.~Ciarlet}, {\em The Finite Element Method for Elliptic Problems},
  North-Holland, Amsterdam, 1978.

\bibitem{Dauge92}
{\sc M.~Dauge}, {\em Neumann and mixed problems on curvilinear polyhedra},
  Integr. Equ. Oper. Theory,  (1992), pp.~227--261.

\bibitem{Deckelnick-Hinze2010}
{\sc K.~Deckelnick and M.~Hinze}, {\em A note on the approximation of elliptic
  control problems with bang-bang controls}, Comput. Optim. Appls., 51 (2012),
  pp.~931--939.

\bibitem{Felgenhauer2003}
{\sc U.~Felgenhauer}, {\em On stability of bang-bang type controls}, SIAM
  Journal on Control and Optimization, 41 (2003), pp.~1843--1867.

\bibitem{Grisvard85}
{\sc P.~Grisvard}, {\em Elliptic Problems in Nonsmooth Domains}, Pitman,
  Boston-London-Melbourne, 1985.

\bibitem{DDW75}
{\sc J.~J.~Douglas, T.~Dupont, and L.~Wahlbin}, {\em The stability in ${L}^q$
  of the ${L}^2$ projection into finite element function spaces}, Numer. Math.,
  23 (1975), pp.~193--197.

\bibitem{MaurerOsmolovskii2003}
{\sc H.~Maurer and N.~P. Osmolovskii}, {\em Second order optimality conditions
  for bang-bang control problems}, Control and Cybernetics, 32 (2003),
  pp.~555--584.

\bibitem{MaurerOsmolovskii2004}
\leavevmode\vrule height 2pt depth -1.6pt width 23pt, {\em Second order
  sufficient conditions for time-optimal bang-bang control}, SIAM Journal on
  Control and Optimization, 42 (2004), pp.~2239--2263 (electronic).

\bibitem{MeyerPanizziSchiela2011}
{\sc C.~Meyer, L.~Panizzi, and A.~Schiela}, {\em Uniqueness criteria for the
  adjoint equation in state-constrained elliptic optimal control}, Numerical
  Functional Analysis and Optimization. An International Journal, 32 (2011),
  pp.~983--1007.

\bibitem{MilyutinOsmolovskii1998}
{\sc A.~A. Milyutin and N.~P. Osmolovskii}, {\em Calculus of variations and
  optimal control}, vol.~180 of Translations of Mathematical Monographs,
  American Mathematical Society, Providence, RI, 1998.
\newblock Translated from the Russian manuscript by Dimitrii Chibisov.

\bibitem{Osmolovskii1994}
{\sc N.~P. Osmolovski{\u\i}}, {\em Quadratic conditions for nonsingular
  extremals in optimal control (a theoretical treatment)}, Russian Journal of
  Mathematical Physics, 2 (1994), pp.~487--516 (1995).

\bibitem{OsmolovskiiMaurer2005}
{\sc N.~P. Osmolovskii and H.~Maurer}, {\em Equivalence of second order
  optimality conditions for bang-bang control problems. {I}. {M}ain results},
  Control and Cybernetics, 34 (2005), pp.~927--950.

\bibitem{OsmolovskiiMaurer2007}
\leavevmode\vrule height 2pt depth -1.6pt width 23pt, {\em Equivalence of
  second order optimality conditions for bang-bang control problems. {II}.
  {P}roofs, variational derivatives and representations}, Control and
  Cybernetics, 36 (2007), pp.~5--45.

\bibitem{PapageorgiouKyritsi-Yiallourou2009}
{\sc N.~S. Papageorgiou and S.~T. Kyritsi-Yiallourou}, {\em Handbook of applied
  analysis}, vol.~19 of Advances in Mechanics and Mathematics, Springer, New
  York, 2009.

\bibitem{Rannacher76}
{\sc R.~Rannacher}, {\em Zur {$L^{\infty }$}-{K}onvergenz linearer finiter
  {E}lemente beim {D}irichlet-{P}roblem}, Math. Z., 149 (1976), pp.~69--77.

\bibitem{Raviart-Thomas83}
{\sc P.~Raviart and J.~Thomas}, {\em Introduction \`a L'analyse Num\'erique des
  Equations aux D\'eriv\'ees Partielles}, Masson, Paris, 1983.

\bibitem{Schatz98}
{\sc A.~Schatz}, {\em Pointwise error estimates and asymptotic error expansion
  inequalities for the finite element method on irregular grids: {P}art {I}.
  {G}lobal estimates}, Math. Comp., 67 (1998), pp.~877--899.

\bibitem{Stampacchia65}
{\sc G.~Stampacchia}, {\em Le probl\`{e}me de {D}irichlet pour les \'equations
  elliptiques du second ordre \`a coefficients discontinus}, Ann. Inst. Fourier
  (Grenoble), 15 (1965), pp.~189--258.

\bibitem{Troltzsch2010}
{\sc F.~{Tr\"{o}ltzsch}}, {\em Optimal Control of Partial Differential
  Equations: Theory, Methods and Applications}, vol.~112 of Graduate Studies in
  Mathematics, American Mathematical Society, Philadelphia, 2010.

\bibitem{Wachsmuth2015}
{\sc D.~Wachsmuth}, {\em Robust error estimates for regularization and
  discretization of bang-bang control problems}, Comput. Optim. Appl., 62
  (2015), pp.~271--289.

\end{thebibliography}
\end{document}